\documentclass[3p,times]{elsarticle}
\usepackage[colorlinks,linkcolor=blue,citecolor=blue]{hyperref}
\usepackage{bookmark}
\usepackage{bm}
\usepackage{amsmath,amsfonts,amsmath,amssymb,amsthm,cases}
\usepackage{algorithm}
\usepackage{algpseudocode} 
\usepackage{multirow}
\usepackage{longtable}
\usepackage{fullpage}
\usepackage{cleveref}
\usepackage{graphics,graphicx,epsfig,epstopdf,subfigure}
\usepackage{color,natbib}
\graphicspath{{jpg/}}

\newtheorem{theorem}{Theorem}[section]
\newtheorem{lemma}{Lemma}[section]

\newtheorem{remark}{Remark}[section]
\newtheorem{example}{Example}[section]

\newtheorem{assumption}{Assumption}[section]
\numberwithin{equation}{section}

\newcommand{\dd}{\,{\rm d}}

\begin{document}
	
\begin{frontmatter}
	\title{Efficient Exponential Integrator Finite Element Method for \\Semilinear Parabolic Equations}
	\author[sjtu]{Jianguo Huang\fnref{hjgfootnote}}
	\ead{jghuang@sjtu.edu.cn}
	\address[sjtu]{School of Mathematical Sciences and MOE-LSC, Shanghai Jiao Tong University, Shanghai 200240, China}
	\fntext[hjgfootnote]{J. Huang's research was partially supported by National Natural Science Foundation of China under grant number  11571237. L. Ju's research was partially supported by US National Science Foundattion  under grant number DMS-2109633 and by US Department of Energy under grant number DE-SC0020270.}
	\author[scu]{Lili Ju\fnref{jllfootnote}}
	\ead{ju@math.sc.edu}
	\address[scu]{Department of Mathematics, University of South Carolina, Columbia, SC 29208, USA }
	\fntext[jllfootnote]{L. Ju's research was partially supported by US National Science Foundattion  under grant number DMS-2109633 and by US Department of Energy under grant number DE-SC0020270.}
	\author[sjtu]{Yuejin Xu}
	\ead{xiaotianma@sjtu.edu.cn}
	\begin{abstract}
		In this paper, we propose  an efficient exponential integrator finite element  method for solving a class of semilinear parabolic equations in rectangular domains. The proposed method first performs the spatial discretization of the model equation using the finite element approximation with  continuous multilinear rectangular basis functions, and then takes the explicit exponential Runge-Kutta approach for time integration of the resulting semi-discrete system to produce fully-discrete numerical solution. Under certain regularity assumptions, error estimates  measured in $H^1$-norm are successfully derived for the proposed schemes with one and two RK stages. More remarkably,  the mass and coefficient matrices of the proposed method can be simultaneously diagonalized with an orthogonal matrix, which provides a fast solution process based on  tensor product spectral decomposition and fast Fourier transform. Various numerical experiments in two and three dimensions are  also carried out to validate the theoretical results and demonstrate the excellent performance of the proposed method.
		
	\end{abstract}
\end{frontmatter}	
	
\section{Introduction}
In this paper, we are devoted to studying numerical solution of the semilinear parabolic equation taking the following form:
\begin{equation}
\label{eq1-1}
u_t = D\Delta u +f(t,u), \quad\;\; \bm x \in \Omega, \; t\geq 0,
\end{equation}
where $\Omega$ is an open rectangular domain in $\mathbb{R}^d$, $D>0$ is the constant diffusion coefficient, $u(t,\bm x)$ is the unknown function and $f(t,u)$ is the nonlinear reaction term of the underlying system. It is well known that the equation \eqref{eq1-1} together with appropriate initial value and boundary condition has been widely used in mathematical models for various scientific and engineering applications, such as the Allen-Cahn equation and some other phase field models \cite{DuFeng2020} for describing the phase transition and separation \cite{Acta1979}, the time-dependent advection diffusion equation and the Navier-Stokes equations for fluid dynamics \cite{Roger1984}, and the Ginzburg-Landau equation \cite{DGP1992} for superconductivity.
Many related numerical methods, especially for its temporal discretization, have been proposed and analyzed  in the past few decades. Some typical ones include
fully-implicit scheme\cite{FePr2003} implicit-explicit method \cite{AscherRuuth1995,WWL2009}, integrating factor  (IF) method \cite{Lawson1967,NZZ2006,WZN2014}, split step method \cite{BoydJohn2001, SanzCalvo1994}, sliders methods \cite{Driscoll2002},  exponential time differencing  (ETD) method \cite{CoxMatthews2002,BorislavWill2005}, invariant energy quadratization (IEQ) method \cite{Yang2016},
integrating factor Runge-Kutta (IFRK) method \cite{IGG2018,LLJF2021}, scalar auxiliary variable (SAV) method \cite{ShXuYa2019} and so on.

The  IF, ETD and IFRK methods fall into the category of exponential integrator-based methods and have received much attention in the past two decades due to their effectiveness and stability in handling stiff semilinear systems, such as the equation  \eqref{eq1-1} with very small $D$ and highly nonlinear $f$.  All the three methods can maintain the exponential behavior of the system in some extent but also have significant differences.
The IF and ETD  methods both start with the Duhamel formula, but they differ  on the way of evaluating the time integration of the terms resulting from multiplication of the exponential integrator and  $f$. 
The IF method uses quadrature rules directly to the whole integrand, while the ETD method first approximates the nonlinear part using polynomial interpolation and then performs exact integration on the resulting integrands. Instead,  the IFRK method first applies an exponential-type variable transformation to eliminate the linear term from the equation and then use the standard Runge-Kutta method for time integration of the transformed system. 
The $L^2$-stability and $L^{\infty}$-stability of ETD methods were established in \cite{DuZhu2004} by means of Fourier transformation (FFT) and maximum bound principle of the underlying equations. Some researches have  been done to  further enhance the stability of ETD method \cite{DuZhu2004, KassamTrefethen2005, MohebbiDehghan2010} and improve their numerical advantages by exploring the stability factor numerically \cite{CoxMatthews2002}.
Motivated by \cite{Pope1963},    a family of exponential Runge-Kutta methods \cite{HochbruckOstermann2005a,HochbruckOstermann2005b,HochbruckOstermann2010,LUAN2014} were systematically developed  with corresponding  consistency and order conditions based on the theory of semigroup.  Exponential integrator-based  methods have also been extended to various areas, such as solution of systems with nondiagonal operators \cite{KassamTrefethen2005}, Rosenbrock-type method \cite{HochbruckOstermann2008}, low-regularity integrators \cite{RoSc2021}, adaptive time stepping \cite{WhalenBrio2015}, and general order parabolic equations \cite{HuangJu2019b,HuangJu2019a}.

A significant feature of the exponential integrator-based methods  is that they all require evaluations of the products of matrix exponentials and vectors. Such operations generally are implemented with Krylov subspace method \cite{DuJuLiQiao2021}, which is still time-consuming for large scale systems.
To overcome this difficulty, some fast FFT-based  algorithms were proposed for efficient implementation of ETD methods
in rectangular domains, but  they only work for the case of using  finite difference method for spatial discretization
\cite{JuZhang2015,ZhuJu2016,DJLQ2019}. The key idea is to utilize tensor product representation of discrete spatial operators and spectral decomposition in each dimension. Moreover, convergence analysis of ETD methods  to the equation  \eqref{eq1-1} highly depends on discrete maximum bound principles, which does not always hold for the classic finite element  or pseudo-spectral discretizations, see  \cite{DuJuLiQiao2021,JuLi2021} and references cited therein.

In this paper, we propose an efficient exponential integrator finite element (EIFE) method for solving the semilinear parabolic equation \eqref{eq1-1} in rectangular domains. In the proposed method, we first carry out the spatial discretization by using continuous multilinear rectangular finite elements to obtain a semi-discrete (in space) system, then apply the explicit exponential Runge-Kutta approach for time integration of the resulting system to achieve the fully-discrete scheme.  Error estimates measured  in $H^1$-norm are also successfully derived for the EIFE  schemes with one and two RK stages
for the problem with Dirichlet boundary condition. Specifically, we first rewrite the semi-discrete system as a finite dimensional spatial operator equation and then error estimate of the semi-discrete solution is naturally obtained by using  the variational formulation and the energy  method.  Next, the fully-discrete system is regarded as a finite dimensional evolution equation, we then adapt and verify the arguments from \cite{HochbruckOstermann2005b}  to estimate the errors of the fully-discrete solutions.
Furthermore, it is observed that the resulting mass and coefficient matrices in the EIFE method can be diagonalized simultaneously with an orthogonal matrix, whose multiplication with a vector can be efficiently computed using FFT, which is held for both Dirichlet and periodic boundary condition cases. Consequently, following \cite{JuZhang2015}, fast solution algorithms are able to be constructed based on the tensor product spectral decomposition.
To the best of our knowledge, the work presented in this paper is the first study on fast numerical method with rigorous fully-discrete error estimates, which combines exponential integrator in time and  finite element discretization in space.

The rest of the paper is organized as follows.  The  EIFE method is first proposed in Section \ref{algorithm-description}, and then its  fully-discrete error analysis is given in Section \ref{semilinear_theory}. The FFT-based fast implementation of EIFE method is illustrated in Section \ref{fastimpl}. In Section \ref{numerical}, various numerical experiments are carried out to validate the theoretical results and demonstrate the excellent performance of the EIFE method. Finally, some concluding remarks are drawn in Section \ref{seccon}.

\section{An exponential integrator finite element method}
\label{algorithm-description}

In this section, we will propose an exponential integrator finite element method  for solving the model equation \eqref{eq1-1}.
Let us start with some standard notations for later representations. Given a bounded Lipschitz domain $G\subset \mathbb{R}^{d}$
and a non-negative integer $s\ge0$, denote by $H^s(G)$ the standard Sobolev spaces on $G$
with the norm $\|\cdot\|_{s,G}$ and the semi-norm $|\cdot|_{s,G}$,  by $(\cdot,\,\cdot)_{G}$ the $L^2$-inner product on $G$, and by $H_0^s(G)$  the closure of $C_0^{\infty}(G)$ with respect to the norm $\|\cdot\|_{s,G}$.  {We also define
	$ \|v\|_{k,\infty,G}={\rm ess\ sup}_{|\bm \alpha| \leq k}\|D^{\bm \alpha}v\|_{L^{\infty}(\Omega)}
	$
	for any function $v$ such that the right-hand side term makes sense, where $\bm \alpha=(\alpha_1,\cdots,\alpha_d)$ is a multi-index and $|\bm\alpha|=\alpha_1+\cdots+\alpha_d$.}  We may omit the subscript for $G=\Omega$ when there is no confusion caused. For any non-negative integer $\ell\ge 0$, $\mathbb{P}_{\ell}(G)$ stands for the set of all polynomials on $G$ with the total degree at most $\ell$. Moreover, given two quantities $a$ and $b$, ``$a\lesssim{b}$" abbreviates ``$a\leq{C}b$", where $C$ is a generic positive constant; ``$a\eqsim b$" is equivalent to ``$a\lesssim b\lesssim a$".

Let us consider the semilinear parabolic equation \eqref{eq1-1} with homogeneous Dirichlet boundary condition and an initial configuration  $u_0\in H^2(\Omega)\cap H_0^1(\Omega)$, that is
\begin{equation}
\label{eq3-1}
\left\{\begin{split}
&u_t = D\Delta u +f(t,u), \quad\;\; \bm x \in \Omega, \;\;  0  \leq t \leq    T ,\\
&u( 0 , \bm x)=u_0 (\bm x), \qquad\quad\, \bm x \in \Omega,\\
&u(t, \bm x) = 0, \qquad \qquad\quad\; \bm x\in {\partial \Omega},\;\; 0  \leq t \leq    T ,\;\;
\end{split}
\right.
\end{equation}
where  $T\ge 0$ is the terminal time.

\subsection{Semi-discretization in space by finite element approximation}
The variational formulation of the above problem \eqref{eq3-1} is to find $u \in L^2( 0 ,   T ;H_0^1(\Omega))$ and $u_t\in L^2( 0 ,   T ; L^2(\Omega))$ such that
\begin{equation}
\label{eq3-10}
\left\{
\begin{split}
&(u_{t}, v) + a(u,v)=(f(t,u), v), \quad \quad  \forall\, v \in H_0^1(\Omega),\; 0  \leq t \leq    T ,\\
&u( 0 )=u_0,
\end{split}
\right.
\end{equation}
where the bilinear operator $a(\cdot,\cdot)$ is symmetric and defined by
\begin{equation}
\label{bilinear-form}
a(w,v)=\int_{\Omega}D\nabla w \cdot \nabla v \dd \bm x,\quad  \forall\,w, v\in H_0^1(\Omega).
\end{equation}

We now define a finite element  space $V_h$ for approximation of $H_0^1(\Omega)$. Since $\Omega\in {\mathbb R}^{d}$ is a rectangular  domain, let us assume  $\overline{\Omega}:=\prod_{i=1}^d [a_i,b_i]$. For each $i= 1,\cdots, d$, we make a uniform partition of $[a_i,b_i]$  with the subinterval size $h_i = \frac{b_i-a_i}{N_i}$, to get the nodes $x_i^j = a_i+jh_i, j=0, \cdots, N_i$ as
\[
a_i = x_i^0 <x_i^1< \cdots<x_i^{N_i}=b_i.
\]
With this regular partition, we obtain a one-dimensional continuous piecewise linear finite element space for $[a_i,b_i]$ as
\begin{eqnarray*}
	V_{h_i}^i(a_i,b_i)&:=&\left\{v\in \mathcal{C}[a_i,b_i]:\ v|_{\big[x_i^{j-1},x_i^j\big]}\in \mathbb{P}_1(\big[x_i^{j-1},x_i^j\big]),1 \leq j \leq N_i \right\}\cap H_0^1(a_i,b_i)\\
	&=&{\rm span}\left\{\phi_i^1(x_i), \cdots, \phi_i^{N_i-1}(x_i) \right\},\nonumber
\end{eqnarray*}
where $\phi_i^j(x_i)$ is the $j$-th nodal basis function of $V_h^i(a_i,b_i)$.
By using the tensor product of all above finite element spaces, we can obtain  a finite element space  for $\Omega$ as follows:
\begin{eqnarray}
\label{tensor-FES}
V_h&:=&V_{h_1}^1(a_1,b_1) \otimes \cdots \otimes V_{h_d}^d(a_d, b_d)\\
&=& {\rm span}\{ {\phi_1^{i_1}}(x_1)\cdots \phi_d^{i_d}(x_d):\ 1\le i_1\le N_1-1,\cdots,1\le i_d\le N_d-1 \}.\nonumber
\end{eqnarray}
It is evident that $V_h\subset H_0^1(\Omega)$.  Define $h=\textstyle\max_{1\le i\le d} h_i$ as the mesh size of the corresponding uniformly rectangular partition $\mathcal{T}_h$ for generating $V_h$. For the forthcoming error analysis, we assume the partition $\mathcal{T}_h$  is quasi-uniform, i.e., $h\eqsim h_i$ for all $1\le i \le d$.
The finite element approximation in space for \eqref{eq3-10} is then to find $u_h \in L^2( 0 ,   T ; V_h)$ such that
\begin{equation}
\label{eq3-11}
\left\{\begin{split}
& (u_{h,t},v_h)  + a(u_h,v_h) = (f(t,u_h),v_h), \qquad \forall\, v_h \in V_h,\; 0  \leq t \leq    T ,\\
&u_h( 0 )=P_hu_0,
\end{split}
\right.
\end{equation}
where $P_h:L^2(\Omega) \rightarrow V_h$ is the $L^2$ orthogonal projection operator. Since $\mathcal{T}_h$  is quasi-uniform, it is easy to show, using the similar arguments in \cite{BrambleXu1991}, that $P_h$ is stable with respect to $L^2$-norm or $H^1$-norm, i.e., $\|P_h u\|_0\lesssim \|u\|_0$ and  $\|P_h u\|_1\lesssim \|u\|_1$ for any $u\in H_0^1(\Omega)$.
Applying the inverse inequality for finite elements \cite{BrennerScott2008, Ciarlet1978} we know
\begin{equation}
\label{eq1-4}
a(w_h, v_h)\lesssim |w_h|_1 |v_h|_1 \lesssim h^{-2} \|w_h\|_0\|v_h\|_0,\quad  \forall\,w_h, v_h\in V_h.
\end{equation}
where the hidden constants are independent of $h$, which means $a(\cdot,\cdot)$ is a bounded bilinear form over $V_h$ with respect to $L^2$-norm. Therefore, by the Riesz representation theorem, there exists a bounded linear operator $L_h: V_h\rightarrow V_h$ such that
\begin{equation}
\label{A-h}
a(w_h,v_h)= (L_h w_h, v_h), \quad  \forall\,w_h, v_h\in V_h.
\end{equation}
Making use of the projection operator $P_h$, we can finally reformulate the problem (\ref{eq3-11}) as the following equivalent semi-discrete (in space) system:
\begin{equation}
\label{eq1-5}
\left\{
\begin{split}
&u_{h,t}+L_hu_h=P_hf(t,u_h), \quad\; \bm x \in \Omega, \;\;  0  \leq t \leq    T ,\\
&u_h( 0 )=P_h u_0, \qquad\qquad\qquad\; \bm x \in \Omega.
\end{split}
\right.
\end{equation}

\subsection{Explicit exponential integrator in time}
Let us divide the time interval $[ 0 ,   T ]$ into $N_T>0$ subintervals
$[t_n,t_{n+1}]$, $n = 0, 1, \cdots, N_T - 1,$
with $\tau_n=t_{n+1}-t_n>0$ being the time step size at $t_n$.
Let $\{e^{-tL_h}\}_{t\geq 0}$ denote  the semigroup on $V_h$  with the infinitesimal generator $(-L_h)$.
For simplicity of presentation, we define $u(t) := u(t,\cdot)$ and  $u_h(t) := u_h(t,\cdot)$.
By the Duhamel principle, the semi-discrete solution $u_h$ to the problem (\ref{eq1-5}) can be equivalently expressed as
\begin{equation}
\label{Duhamel}
u_h(t_{n+1}) = e^{-\tau_nL_h }u_h(t_n)+\int_{0}^{\tau_n} e^{-(\tau_n - \sigma)L_h}P_hf(t_n+\sigma, u_h(t_n+\sigma))\dd \sigma.
\end{equation}
Denote by $ u_h^n$  the fully-discrete numerical solution at the time step $t_n$ after temporal discretization of \eqref{Duhamel}.
We  then apply the classic  explicit exponential Runge-Kutta approach \cite{HochbruckOstermann2010} (see Table \ref{tab1}) to \eqref{Duhamel} and obtain  a fully-discrete numerical method for solving  the semilinear parobolic problem \eqref{eq3-1} as follows: for $n=0,1,\cdots, N_T-1$,
\begin{equation}
\label{eq2-2}
\left\{
\begin{split}
u_h^{n+1} &= \chi(-\tau_nL_h)u_h^n+\tau_n\textstyle\sum\limits_{i=1}^s b_i(-\tau_nL_h)G_{ni},\\
U_{ni} &= \chi_i(-\tau_nL_h)u_h^n + \tau_n\textstyle\sum\limits_{j=1}^{i-1} a_{ij}(-\tau_nL_h)G_{nj},\quad i= 1, \cdots,s,\\
G_{ni} &= P_hf(t_n+c_i\tau_n, U_{ni}),\quad i= 1, \cdots, s,
\end{split}
\right.
\end{equation}
where the integer $s>0$ denotes the number of stages for exponential Runge-Kutta method.


\begin{table}[!htbp]
	\centering
	\caption{Explicit exponential Runge-Kutta tableau for (\ref{Duhamel}) with $s$ denoting the number of stages.}
	\begin{tabular}{c|cccc|c}
		$c_1$ & & & & & $\chi_1(-\tau_n L_h)$\\	
		$c_2$ & $a_{21}(-\tau_n L_h)$ & & & & $\chi_2(-\tau_n L_h)$\\
		$\vdots$ & $\vdots$ & $\ddots$ & & & $\vdots$\\
		$c_s$ & $a_{s1}(-\tau_n L_h)$ & $\cdots$ & $a_{s, s-1}(-\tau_n L_h)$ & & $\chi_s(-\tau_n L_h)$\\
		\hline
		& $b_1(-\tau_n L_h)$ & $\cdots$ & $b_{s-1}(-\tau_n L_h)$ & $b_{s}(-\tau_n L_h)$ & $\chi(-\tau_n L_h)$
	\end{tabular}
	\label{tab1}
\end{table}

Here the interpolation nodes $c_1, \cdots, c_s$ are $s$ distinct nodes selected in $[0, 1]$,  and the weights  are given by
\begin{equation*}
\left\{
\begin{split}
b_i(-\tau_n L_h) &= \int_{0}^1 e^{-\tau_n(1-\theta)L_h}l_i(\theta)\dd \theta,\quad i= 1, \cdots, s,\\
a_{ij}(-\tau_nL_h) &= \frac{1}{\tau_n}\int_0^{c_i\tau_n}e^{-(c_i\tau_n-\tau)L_h}l_j(\tau)\dd \tau,\quad i= 1, \cdots, s,\;j=1,\cdots, i-1,
\end{split}
\right.
\end{equation*}
with $\{l_i(\theta)\}_{i=1}^s$ being the classic Lagrange interpolation polynomials
$$l_i(\theta):=\prod\limits_{m=1,m \neq i}^s \frac{\theta-c_m}{c_i-c_m}, \quad i = 1 , \cdots s. $$
The coefficients $\chi, \chi_i$ are constructed from exponential functions or approximations of such functions evaluated at the operator {$(-\tau_n L_h)$}. Let us define
$$
\varphi_k(z) = \int_0^1e^{(1-\theta)z}\frac{\theta^{k-1}}{(k-1)!}\dd\theta, \;\; k\geq 1.
$$
It is easy to check that $\varphi_0(z) =e^z$ and $\varphi_{k+1}(z) = \frac{\varphi_k(z)-\varphi_k(0)}{z}$ for $k>0$.
Then,
$$
\qquad \varphi_k(-\tau_n L_h) = \frac{1}{\tau_n^k} \int_0^{\tau_n} e^{-(\tau_n-\tau)L_h}\frac{\tau^{k-1}}{(k-1)!}\dd \tau, \;\;k\geq 1.$$

To ensure the consistency of the explicit exponential Runge-Kutta method, 
it  is always assumed that the following conditions hold \cite{HochbruckOstermann2010}:
\begin{equation}
\label{eq2-3}
\left\{
\begin{split}
\chi(-\tau_n L_h) = e^{-\tau_n L_h}, \qquad&\\
\qquad\chi_1(-\tau_n L_h) = 1, \qquad& \chi_i(-\tau_n L_h)=e^{-c_i \tau_n L_h},\quad 1 \leq i \leq s,\\
\textstyle\sum\limits_{j=1}^s b_j(-\tau_n L_h) = \varphi_1(-\tau_n L_h), \qquad & \textstyle\sum\limits_{j=1}^{i-1} a_{ij}(-\tau_n L_h) = c_i \varphi_1(-c_i\tau_n L_h), \;\; 1 \leq i \leq s.
\end{split}
\right.
\end{equation}

With the help of the above consistency condition (\ref{eq2-3}), we can find $c_1=0$ and {obtain the explicit forms of the functions $\chi$ and $\chi_i$ in (\ref{eq2-2})}. Finally, the proposed fully-discrete scheme  \eqref{eq2-2} can be rewritten as
\begin{equation}
\label{fully}
\left\{
\begin{split}
u_h^{n+1} &= e^{-\tau_nL_h}u_h^n+\tau_n\textstyle\sum\limits_{i=1}^s b_i(-\tau_nL_h)G_{ni},\\
U_{ni} &= e^{-c_i\tau_nL_h}u_h^n+\tau_n\textstyle\sum\limits_{j=1}^{i-1} a_{ij}(-\tau_nL_h)G_{nj}, \quad i= 1, \cdots,s,\\
G_{ni} &= P_hf(t_n+c_i\tau_n,U_{ni}),\quad i= 1, \cdots,s.
\end{split}
\right.
\end{equation}


We refer to the above fully-discrete scheme  \eqref{fully}  as the exponential integrator finite element (EIFE) method. In particular, if the number of  RK stages $s=1$, then  the so-called first-order Euler exponential scheme  is obtained for time stepping, and (\ref{fully}) can correspondingly be expressed as
\begin{equation}
\label{Euler}
u_h^{n+1} = e^{-\tau_n L_h}u_h^n + \tau_n\varphi_1(-\tau_nL_h)P_hf(t_n,u_h^n),
\end{equation}
which is called EIFE1.
If the number of  RK stages $s=2$, then the two interpolation nodes are taken as $c_1=0$ and $c_2 \in (0, 1]$ and the
two-stage second-order exponential Runge-Kutta scheme is obatined for time stepping,   and (\ref{fully}) correspondingly reads
\begin{equation}
\label{second-order}
\begin{split}
&u_h^{n+1} =\; e^{-\tau_nL_h}u_h^n+\tau_n\big((\varphi_1(-\tau_nL_h)-\frac{1}{c_2}\varphi_2(-\tau_nL_h))P_hf(t_n,u_h^n)\\
&+\frac{1}{c_2}\varphi_2(-\tau_nL_h)P_hf(t_n+c_2\tau_n,e^{-c_2\tau_nL_h}u_h^n+c_2\tau_n\varphi_1(-c_2\tau_nL_h)P_hf(t_n,u_h^n)) \big),
\end{split}
\end{equation}
which is denoted as EIFE2.

Fully-discrete error analysis of \eqref{Euler} and \eqref{second-order} will be carefully studied in Section \ref{semilinear_theory}.
For even higher order ($\geq 3$) explicit exponential Runge-Kutta schemes, more complicated order conditions are needed and we can refer to \cite{HochbruckOstermann2010}  for details.
It is worth noting that the proposed EIFE method is so far described in an abstract setting for the convenience of further theoretical analysis. If we express the finite element numerical solution in terms of nodal basis functions of $V_h$, we can further reformulate the equations \eqref{eq3-11} and \eqref{eq1-5} in the matrix form. More remarkably, the resulting mass and coefficient matrices can be diagonalized simultaneously with an orthogonal matrix and  the product of the $e^{-\tau_n L_h}$ with any vector consequently can be realized efficiently using  tensor product spectral decomposition and FFT, which  render a fast  implementation for the
EIFE method \eqref{fully} as described in Section \ref{fastimpl}.

\section{Error analysis of the EIFE method}\label{semilinear_theory}

In this section, we focus on error estimates of fully-discrete solutions produced by the EIFE method \eqref{fully}  for solving the semilinear parabolic problem \eqref{eq3-1} (i.e., the case with homogeneous Dirichlet boundary condition). Note that the analysis also work for the problem with nonhomogeneous Dirichlet boundary condition after minor modifications. From now on, we always assume the  dimension of the problem $d\leq 3$.

\subsection{Some preliminary lemmas}
We  first present some results related to  the semigroup $\{e^{-tL_h}\}_{t\geq 0}$ and some terms used in the exponential Runge-Kutta schemes, which are important to the forthcoming analysis of the proposed EIFE method.

\begin{lemma}
	\label{lemma1}
	\begin{itemize}
		\item[\rm (i)] For any given  parameter $\gamma\geq 0$, it holds
		\begin{equation}
		\label{eq1-8}
		\|e^{-\tau L_h}\|_0 + \|\tau^{\gamma}L_h^{\gamma}e^{-\tau L_h}\|_0 \lesssim 1, \quad \forall\,\tau > 0,\, \forall\, h>0.
		\end{equation}
		\item[\rm (ii)]  For any given parameter $0 \leq \gamma \leq 1$, it holds
		\begin{equation}\label{555}
		\Big\|\tau L_h^{\gamma} \textstyle\sum\limits_{j=1}^{n-1}e^{-j\tau L_h}\Big\|_0 \lesssim 1, \quad\forall\,\tau > 0,\, \forall\, h>0.
		\end{equation}
		\item[\rm (iii)] For any given parameter $0
		\leq\gamma\leq 1$, it holds
		\begin{equation}\label{666}
		\|\phi(-\tau L_h)\|_0+\|\tau^{\gamma}L_h^{\gamma}\phi(-\tau L_h)\|_0 \lesssim 1,\quad \forall\,\tau > 0,\, \forall\, h>0,
		\end{equation}
		where $\phi(-\tau L_h)=b_i(-\tau L_h)$ or $\phi(-\tau L_h)=a_{ij}(-\tau L_h)$,  $i,j=1, \cdots, s$.
	\end{itemize}
\end{lemma}

\begin{proof}
	(i) Since $a(\cdot,\cdot)$ is a symmetric bilinear operator, we easily know $L_h$ is a linear symmetric operator on $V_h$. On the other hand, by the Poincar\'{e} inequality, there exists a constant $\alpha>0$ such that
	\[
	\alpha \|v_h\|_0^2\le a(v_h,v_h),\quad \forall\,v_h\in V_h.
	\]
	This combined with \eqref{eq1-4} gives $0\leq \alpha \leq \lambda \lesssim h^{-2}$ for all $\lambda\in \lambda(L_h)$ by means of Rayleigh representation theorem, where $\lambda(L_h)$ denotes the set of all eigenvalues of $L_h$. Therefore, we obtain
	\begin{equation}
	\label{semigroup-1}
	\|e^{-\tau L_h }\|_{0}\leq e^{-\alpha \tau} < 1.
	\end{equation}
	On the other hand, $L_h$ is a symmetric operator on $V_h$, so is $\tau^{\gamma}L_h^{\gamma}e^{-\tau L_h}$. Thus it holds
	\begin{equation}
	\label{semigroup-2}
	\|\tau ^{\gamma}L_h^{\gamma}e^{-\tau L_h}\|_0=\textstyle\max_{\lambda \in \lambda(L_h)} |(\tau \lambda)^{\gamma}e^{-\tau\lambda}|.
	\end{equation}	
	Let us consider an auxiliary function
	$
	g(x)=x^{\gamma}e^{-x}$ for $x\ge 0$.
	The derivative with respect to $x$ is $g^{\prime}(x)=x^{\gamma-1}e^{-x}(\gamma-x)$, so the maximum of $g(x)$ is taken at $x=\gamma$, which implies
	\begin{equation}\label{ttt}
	g(x)\le \gamma^{\gamma}e^{-\gamma},\quad  \forall\,x\ge 0.
	\end{equation}
	Combination of \eqref{ttt} with \eqref{semigroup-2}  immediately gives us
	$
	\|\tau ^{\gamma}L_h^{\gamma}e^{-\tau L_h}\|_0\le \gamma^{\gamma}e^{-\gamma},
	$
	which together with  \eqref{semigroup-1} then directly deduces  \eqref{eq1-8}.
	
	(ii) Follow the similar arguments for deriving \eqref{semigroup-2},
	\begin{equation*}
	\begin{split}
	\Big\|\tau L_h^{\gamma} \textstyle\sum\limits_{j=1}^{n-1} e^{-j\tau L_h}\Big\|_0 &
	=\textstyle\max\limits_{\lambda \in \lambda(L_h)}\Big|\tau \lambda^{\gamma} \textstyle\sum\limits_{j=1}^{n-1} e^{-j\tau \lambda}\Big|
	\leq \textstyle\max\limits_{\lambda \in \lambda(L_h)} \lambda^{\gamma -1}\Big|\frac{\tau \lambda}{e^{\tau \lambda}-1}\Big|.
	\end{split}
	\end{equation*}
	Since the auxiliary function $g(x) = \Big|\frac{x}{e^{x}-1}\Big|\leq 1$ for any $x>0$, $0\leq \gamma \leq 1$, we have
	\begin{equation}
	\Big\|\tau L_h^{\gamma} \textstyle\sum\limits_{j=1}^{n-1}e^{-j\tau L_h}\Big\|_0 \lesssim  1.
	\end{equation}
	
	(iii) If $\phi(x) = b_i(x)$ for $1\leq i\leq s$, then
	\begin{equation*}
	\begin{split}
	\|\phi(-\tau L_h)\|_0 &= \Big\|\int_0^1 e^{-\tau(1-\theta) L_h}l_i(\theta)\dd \theta\Big\|_0
	\leq \int_0^1 \|e^{-\tau(1-\theta)L_h}\|_0 |l_i(\theta)|\dd \theta \lesssim 1,
	\end{split}
	\end{equation*}
	since the interpolation nodes $c_1, \cdots, c_s$ are  fixed.	
	For all $\gamma\in [0,1]$, follow the similar arguments, we have
	\begin{equation}\label{sss}
	\begin{split}
	\|\tau^{\gamma}L_h^{\gamma}\phi(-\tau L_h)\|_0  &= \textstyle\max\limits_{\lambda \in \lambda(L_h)}\Big|\tau^{\gamma}\lambda^{\gamma}\int_0^1 e^{-\tau(1-\theta)\lambda}l_i(\theta)\dd \theta\Big|\\
	&\lesssim \textstyle\max\limits_{\lambda \in \lambda(L_h)} \Big|\tau^{\gamma}\lambda^{\gamma}\int_0^1 e^{-\tau(1-\theta)\lambda}\dd \theta\Big|=\textstyle\max\limits_{\lambda \in \lambda(L_h)} (\tau \lambda)^{\gamma-1}(1-e^{-\tau\lambda}).
	\end{split}
	\end{equation}
	Let us consider the auxiliary function $g(x) = x^{\gamma - 1}(1-e^{-x})$ for $x > 0$ where $0\leq \gamma\leq 1$.  Since $\lim\limits_{x \to 0^+}f(x)=0$ and $\lim\limits_{x\to +\infty} f(x)=0$, the function $f(x)$ must take its maximum at a point $x_0\in (0, +\infty)$, and  $f^{\prime}(x_0)=0$ by the Fermat lemma.  Note that
	\[
	f^{\prime}(x)=x^{\gamma-2}e^{-x}((\gamma-1)(e^x-1)+x),
	\]
	which implies that $(\gamma-1)(e^{x_0}-1)+x_0=0$, and $x_0$ depends only on $\gamma$. Hence, $f(x)$ is bounded by a constant $C(\gamma)>0$. This combined with the above estimate \eqref{sss} gives
	\begin{equation}
	\|\tau^{\gamma}L_h^{\gamma}\phi(-\tau L_h)\|_0 \lesssim 1.
	\end{equation}
	Similarly,the above results are also held for $\phi(x) = a_{ij}(x)$, $i,j=1, \cdots, s$.
\end{proof}

{In addition, recalling the definitions \eqref{bilinear-form} and \eqref{A-h} and noting that $L_h$ is a symmetric positive definite operator, we have the following important property:
	\begin{equation}
	\label{norm-relation}
	\|v\|_1\eqsim \|L_h^{\frac 1 2 } v\|_0,\quad \forall \ v\in V_h,
	\end{equation}
	which establishes an important relation between the $H^1$ and $L^2$ norms over $V_h$, and will be frequently used later on.}
Next let us introduce the mild growth condition for the function $f$ as given in \cite{ThomeeVidar2006} and some regularity conditions required for the exact solution $u(t)$ in order  to carry out convergence and error analysis of EIFE method.

\begin{assumption}
	\label{assumption2}
	The function $f(t,\zeta) $  grows mildly with respect to $\zeta$, i.e., there exists a number  $p>0$  for $d = 1, 2$ or $p\in (0, 2]$ for $d=3$  such that
	\begin{equation}
	\label{assumption_eq}
	\Big|\frac{\partial f}{\partial \zeta}(t, \zeta)\Big| \lesssim 1+|\zeta|^p,\quad \forall\,t,\, \zeta \in \mathbb{R}.
	\end{equation}
\end{assumption}
\begin{assumption}
	\label{assumption3}
	The function $f(t,\zeta)$  is sufficiently smooth with respect to $t$, i.e., for any given constant $K>0$, it holds
	\begin{equation}
	\label{assumption_eq2}	
	\textstyle\sum\limits_{|\bm\alpha| \leq 2}\Big|D^{\bm\alpha}f(t,\zeta)  \Big| \lesssim 1, \quad \forall\,t \in [ 0 ,    T ],\zeta \in [-K,K].
	\end{equation}
\end{assumption}

\begin{assumption}
	\label{assumption4}
	The exact solution $u(t)$ satisfies some of the following regularity conditions:
	\begin{subequations}\label{strong-regularity}
		\begin{align}
		\textstyle\sup\limits_{ 0  \leq t \leq    T } \|u(t)\|_{2, \Omega} &\lesssim 1,\label{strong-regularity2}\\
		\textstyle\sup\limits_{ 0  \leq t \leq    T } \|u_t(t)\|_{0, \infty, \Omega} &\lesssim 1, \label{strong-regularity1}\\
		\textstyle\sup\limits_{ 0  \leq t \leq    T } \|u_{tt}(t)\|_{0, \infty, \Omega} &\lesssim 1,\label{strong-regularity3}
		\end{align}
		where the hidden constants may depend on $T$.
	\end{subequations}
\end{assumption}

We then have the following result on the locally-Lipschitz continuity of $f$.
\begin{lemma}
	\label{lemma3}
	Suppose that the function $f$  satisfies Assumption \ref{assumption2}, and the exact solution $u(t)$ fulfills \eqref{strong-regularity2} in Assumptions \ref{assumption4}. Then $f$ is locally-Lipschitz continuous in a strip along the exact solution $u(t)$, i.e., for any given constant $R >0$,
	\begin{equation}\label{rrr}
	\|f(t, v)-f(t, w)\|_0 \lesssim \|v-w\|_1,
	\end{equation}
	for  any $ t\in [ 0 ,    T ]$ and  $v, w\in V_h$ satisfying
	\begin{equation*}
	\label{growth-assumption}
	\textstyle\max\{\|v-u(t)\|_1, \|w-u(t)\|_1 \} \leq R.
	\end{equation*}
\end{lemma}
\begin{proof}
	It follows from \eqref{strong-regularity2} and \eqref{growth-assumption} that
	\begin{equation}
	\label{bound}
	\textstyle\max\{\|v\|_1,\|w\|_1\}\le R+\|u(t)\|_1\lesssim 1+R.
	\end{equation}
	
	On the other hand, by the Lagrange mean value theorem and \eqref{assumption_eq},
	\begin{equation} \label{eq4-10}
	\begin{split}
	\|f(t, v)-f(t,w)\|_0^2 &= \Big\|\frac{\partial f}{\partial u}(t, \xi)(v-w)\Big\|_0^2\leq \int_{\Omega}(1+|\xi|^p)^2|v-w|^2\dd \bm x
	\\
	&\leq \int_{\Omega}(1+|v|^p)^2|v-w|^2\dd \bm x+\int_{\Omega}(1+|w|^p)^2|v-w|^2\dd \bm x,
	\end{split}
	\end{equation}
	where $\xi(\bm x) = \theta(\bm x)v(\bm x)+(1-\theta(\bm x))w(\bm x)$ for some $\theta(\bm x)\in [0, 1]$. It is clear that
	\begin{equation}
	\label{eq4-11}
	\int_{\Omega}(1+|v|^p)^2|v-w|^2\dd  \bm x= \int_{\Omega} |v|^{2p}|v-w|^2 \dd  \bm x + 2\int_{\Omega} |v|^p|v-w|^2 \dd  \bm x + \int_{\Omega} |v-w|^2 \dd  \bm x.
	\end{equation}
	
	It suffices to show  below the bound of the first term in the right-hand side of (\ref{eq4-11}) since the other two terms can be treated similarly.
	
	Case I: $d = 1$ or $2$. In this case, $0 < p < \infty$. We choose  $q_1, q_2$ satisfying
	\begin{equation*}
	\frac{1}{q_1} + \frac{1}{q_2} = 1, \quad 1 \leq q_1 < \infty, \quad 1 \leq q_2 < \infty, \quad 2pq_1 \geq 1.
	\end{equation*}
	Then by using Holder's inequality, Sobolev embedding theorem and \eqref{bound},
	\begin{equation}\label{2-d}
	\begin{split}
	\int_{\Omega} |v|^{2p}|v-w|^2 \dd x
	&\leq  \Big(\int_{\Omega} |v|^{2pq_1} \dd  \bm x \Big)^{\frac{1}{q_1}}\Big(\int_{\Omega}|v-w|^{2q_2}\dd  \bm x \Big)^{\frac{1}{q_2}} 
	\\
	&=\|v\|_{L^{2pq_1}}^{2p}\|v-w\|_{L^{2q_2}}^2
	\lesssim \|v\|_1^{2p}\|v-w\|_1^2\lesssim \|v-w\|_1^2.
	\end{split}
	\end{equation}
	
	Case II: $d=3$. In this case, $0 < p \leq 2$. In view of the Sobolev embedding theorem, $H^1(\Omega) \hookrightarrow L^q(\Omega)$, where $1 \leq q \leq 6$. We select appropriate $q_1, q_2$ satisfying
	\begin{equation*}
	\frac{1}{q_1}+\frac{1}{q_2} = 1, \quad 1 \leq q_1 < \infty, \quad 1 \leq q_2 \leq 3, \quad 1 \leq 2pq_1 \leq 6.
	\end{equation*}
	In fact, these conditions hold if $(q_1,q_2)$ satisfies the conditions
	\begin{equation*}
	\textstyle\max\Big\{ \frac 3 2, \frac{1}{2p} \Big\}\leq q_1 \leq \frac 3 p, \quad 1 \leq q_2 \leq 3.
	\end{equation*}
	Noting $0 < p\leq 2$,  we can find the existence of such a pair $(q_1,q_2)$ in terms of the above conditions. Follow the similar arguments for deriving \eqref{2-d}, we have
	\begin{equation}\label{3-d}
	\begin{split}
	\int_{\Omega} |v|^{2p}|v-w|^2 \dd  \bm x &\leq \Big(\int_{\Omega} |v-w|^{2q_2}\dd  \bm x \Big)^{\frac{1}{q_2}}\Big(\int_{\Omega} |v|^{2pq_1}\dd \bm x \Big)^{\frac{1}{q_1}}
	\\
	&= \|v-w\|_{L^{2q_2}}^2 \|v\|_{L^{2pq_1}}^{2p}\lesssim \|v-w\|_1^2 \|v\|_1^{2p} \lesssim \|v-w\|_1^2.
	\end{split}
	\end{equation}
	
	Based on \eqref{2-d}, \eqref{3-d} and \eqref{eq4-11}, we obtain that
	\begin{equation}\label{dddv}
	\int_{\Omega}(1+|v|^p)^2|v-w|^2\dd  \bm x\lesssim \|v-w\|_1^2,
	\end{equation}
	and similarly
	\begin{equation}\label{dddw}
	\int_{\Omega}(1+|w|^p)^2|v-w|^2\dd  \bm x\lesssim \|v-w\|_1^2,
	\end{equation}
	Finally, the combination of \eqref{eq4-10}, \eqref{dddv} and \eqref{dddw} leads to \eqref{rrr}.  %
\end{proof}

\subsection{Fully-discrete error estimates}
{In the rest of this section, for simplicity we will assume the time partition is uniform, i.e., $\Delta t=\tau_0=\cdots=\tau_{N_T-1}$ and $t_n = n\Delta t$.} Let $u_h(t)$ be the solution of the  semi-discrete (in space) problem \eqref{eq1-5} (or  \eqref{Duhamel}), and $\{u_h^n\}$ the fully-discrete solution produced by the EIFE method \eqref{fully}. For  the error between the exact solution  $u(t)$ and  the fully-discrete solution  $\{u_h^n\}$ measured in the $H^1$-norm, we have by the triangle inequality
\begin{equation}\label{errtri}
\|u(t_n)-u_h^n\|_1 \leq \|u(t_n)-u_h(t_n)\|_1 + \|u_h(t_n)-u_h^n\|_1.
\end{equation}
Thus we will estimate  $\|u(t_n)-u_h(t_n)\|_1$ and  $\|u_h(t_n)-u_h^n\|_1$ respectively.

First of all, the following result  readily comes from Theorem 14.2 in \cite{ThomeeVidar2006}.
\begin{lemma}
	\label{semi-estimate}
	Suppose that the function $f$ satisfies Assumptions \ref{assumption2} and  \ref{assumption3}, and the exact solution $u(t)$ fulfills (\ref{strong-regularity2}) in Assumptions \ref{assumption4}.  There exists a constant $h_0>0$ such that if the spatial mesh size $h \leq h_0$, then
	\begin{equation}
	\label{semi-discrete1}
	\|u_h(t)-u(t)\|_1 \lesssim h, \quad \forall\,t \in [ 0 ,    T ],
	\end{equation}
	where  the hidden constant is independent of $h$.
\end{lemma}

A direct consequence of Lemma \ref{semi-estimate} is
\begin{equation}\label{erra}
\|u(t_n)-u_h(t_n)\|_1\lesssim h, \quad\forall\, n=0,\cdots,N_T,
\end{equation}
where  the hidden constant is independent of $h$. Then the remaining part for us is to  bound $\|u_h(t_n)-u_h^n\|_1$. Note that the semi-discrete solution $u_h(t)$ obtained from (\ref{eq1-5}) itself depends on the spatial mesh size $h$, on the other hand we need remove the dependence of the hidden constants on $h$ in estimating $\|u_h(t_n)-u_h^n\|_1$. In order to achieve this goal, we convert the semi-discrete solution $u_h(t_{n+1})$ ($n=0,1,\cdots,N_T-1$) into the sum of the following two parts for further analysis:
\begin{equation}
\label{eq1-6}
\begin{split}
u_h(t_{n+1}) =\;& e^{-\Delta t L_h }u_h(t_n)+\int_{0}^{\Delta t } e^{-(\Delta t  - \sigma)L_h}P_hf(t_n+\sigma, u_h(t_n+\sigma))\dd \sigma\\
=\;& e^{-\Delta t  L_h}u_h(t_n) + \int_0^{\Delta t } e^{-(\Delta t  - \sigma)L_h}P_hf(t_n+\sigma, u(t_n+\sigma)) \dd \sigma\\
&\hspace{-0.2cm}+\int_0^{\Delta t } e^{-(\Delta t  - \sigma)L_h}\big(P_hf(t_n+\sigma, u_h(t_n+\sigma))-P_hf(t_n+\sigma, u(t_n+\sigma)) \big) \dd \sigma.
\end{split}
\end{equation}

Define the following functions:
\begin{equation}
\left\{
\label{function}
\begin{split}
\psi_i(-\Delta t L_h) &= \varphi_i(-\Delta t L_h)-\textstyle\sum\limits_{k=1}^s b_k(-\Delta t L_h)\frac{c_k^{i-1}}{(i-1)!}, \; i=1, \cdots, s,\\
\psi_{j, i}(-\Delta t L_h) &= \varphi_j(-c_i\Delta t L_h)c_i^j - \textstyle\sum\limits_{k=1}^{i-1} a_{ik}(-\Delta t L_h)\frac{c_k^{j-1}}{(j-1)!}, \; i,j =1, \cdots, s.
\end{split}
\right.
\end{equation}
We also denote $f^{(k)}(t, u(t)) = \frac{d^k}{dt^k}f(t,u(t))$ as the $k$-th full differentiation of $f$ with respect to $t$.
By comparing (\ref{eq1-6}) with the fully-discrete scheme (\ref{fully}), we then obtain
\begin{eqnarray}
u_h(t_n+c_i\Delta t )&= &e^{-c_i \Delta t  L_h}u_h(t_n)+\Delta t  \textstyle\sum\limits_{j=1}^{i-1} a_{ij}(-\Delta t  L_h)\nonumber\\
&& \qquad P_hf(t_n+c_j \Delta t , u(t_n+c_j\Delta t )) + \delta_{ni},\label{111}\\
u_h(t_{n+1}) &= &e^{-\Delta t  L_h}u_h(t_n)+\Delta t  \textstyle\sum\limits_{i=1}^s b_i(-\Delta t L_h)\nonumber\\
&&\qquad P_hf(t_n+c_i\Delta t , u(t_n+c_i\Delta t )) + \delta_{n+1},\label{222}
\end{eqnarray}
where the defect terms $\{\delta_{ni}\}_{i=1}^s$ and  $\delta_{n+1}$ are respectively given by
\begin{equation*}
\begin{split}
\delta_{ni} &= \textstyle\sum\limits_{j=1}^{r} \Delta t ^j \psi_{j,i}(-\Delta t  L_h)P_hf^{(j-1)}(t_n,u(t_n))+\delta_{ni}^{[r]},\\
\delta_{n+1} &= \textstyle\sum\limits_{i=1}^r \Delta t ^i \psi_i(-\Delta t  L_h)P_hf^{(i-1)}(t_n,u(t_n))+\delta_{n+1}^{[r]},\\
\end{split}
\end{equation*}
with the remainders $\delta_{ni}^{[r]}$ and $\delta_{n+1}^{[r]}$ defined respectively by
\begin{equation*}
\begin{split}
\delta_{ni}^{[r]}=&\int_0 ^{c_i \Delta t }e^{-(c_i\Delta t  - \tau)L_h}\int_0^{\tau}\frac{(\tau-\sigma)^{r-1}}{(r-1)!}P_hf^{(r)}(t_n+\sigma, u(t_n+\sigma))\dd \sigma \dd \tau \\
&- \Delta t \textstyle\sum\limits_{k=1}^{i-1}a_{ik}(-\Delta t  L_h)\displaystyle\int_0^{c_k\Delta t }\frac{(c_k\Delta t -\sigma)^{r-1}}{(r-1)!}P_hf^{(r)}(t_n+\sigma, u(t_n+\sigma))\dd \sigma\\
&+ \int_0^{c_i \Delta t } e^{-(c_i\Delta t -\sigma)L_h}\big(P_hf(t_n+\sigma,u_h(t_n+\sigma))-P_hf(t_n+\sigma, u(t_n+\sigma)) \big)\dd \sigma,
\end{split}
\end{equation*}
\begin{equation*}
\begin{split}
\delta_{n+1}^{[r]} =& \int_0 ^{\Delta t }e^{-(\Delta t  - \tau)L_h}\int_0^{\tau}\frac{(\tau-\sigma)^{r-1}}{(r-1)!}P_hf^{(r)}(t_n+\sigma, u(t_n+\sigma))\dd \sigma \dd \tau \\
&- \Delta t  \textstyle\sum\limits_{i=1}^s b_i(-\Delta t  L_h)\displaystyle\int_0^{c_i\Delta t }\frac{(c_i\Delta t -\sigma)^{r-1}}{(r-1)!}P_hf^{(r)}(t_n+\sigma,u(t_n+\sigma))\dd \sigma\\
&+ \int_0^{\Delta t }e^{-(\Delta t -\sigma)L_h}\big(P_hf(t_n+\sigma, u_h(t_n+\sigma))-P_hf(t_n+\sigma, u(t_n+\sigma)) \big)\dd \sigma.
\end{split}
\end{equation*}
Here $r$ can be any nonnegative integers such that $f^{(r)}(t, u(t))$  exists and is continuous.

In what follows, we will adopt the arguments proposed in  \cite{HochbruckOstermann2005b} to bound $\|u_h^n-u_h(t_n)\|_1$.
For brevity, let us define  $e_n = u_h^n-u_h(t_n)$
and  $E_{ni}=U_{ni}-u_h(t_n+c_i\Delta t )$ for $i=1,\cdots,s$.
Then we arrive at the following recurrence relations:
\begin{eqnarray}
&&E_{ni} = e^{-c_i\Delta t L_h}e_n+
\Delta t \textstyle\sum\limits_{j=1}^{i-1} a_{ij}(-\Delta t L_h)\big(P_hf(t_n+c_j\Delta t ,U_{nj})\nonumber\\
&&\hspace{5cm}-P_hf(t_n+c_j\Delta t ,u_h(t_n+c_j\Delta t ))\big)-\delta_{ni}. \label{recursion-E}\\
&&e_{n+1} = e^{-\Delta t L_h}e_n+\Delta t \textstyle\sum\limits_{i=1}^s b_i(-\Delta t L_h)\big(P_hf(t_n+c_i\Delta t ,U_{ni})
\nonumber\\&&\hspace{5cm}-P_hf(t_n+c_i\Delta t ,u_h(t_n+c_i\Delta t ))\big)-\delta_{n+1},\label{recursion-e}
\end{eqnarray}

We first have the following result on the defect terms in \eqref{111} and \eqref{222}.

\begin{lemma}
	\label{lemma4}  Given an integer $r=1$ or $2$.
	Suppose that the function $f$ satisfies Assumptions \ref{assumption2} and \ref{assumption3}, and  the exact solution $u(t)$ fulfills (\ref{strong-regularity2}) and (\ref{strong-regularity1}) in Assumptions \ref{assumption4}. Suppose that $u(t)$ additionally  fulfills (\ref{strong-regularity3})  if $r=2$.
	Then for $n=0,\cdots,N_T$, $i = 1, \cdots, s$, it holds that
	\begin{subequations}\label{lemma4_eq}
		\begin{align}
		\textstyle\|\delta_{ni}^{[r]}\|_1 \lesssim (\Delta t )^{r+1}\textstyle\sup\limits_{0\leq \eta \leq 1}\|f^{(r)}(t_n+\eta \Delta t ,u(t_n+\eta \Delta t ))\|_1 + h,\label{lemma4-1}\\
		\textstyle\Big\|\textstyle\sum\limits_{j=0}^{n-1}e^{-j\Delta t  L_h}\delta_{n-j}^{[r]}\Big\|_1 \lesssim (\Delta t)^r \textstyle\sup\limits_{ 0  \leq t \leq    T } \|f^{(r)}(t,u(t))\|_1 +h.\label{lemma4-2}
		\end{align}
		Note that the above hidden constants are independent of $h$ and $\Delta t$.	
	\end{subequations}
\end{lemma}
\begin{proof}
	Recalling \eqref{eq1-8} in Lemma \ref{lemma1} and the relation \eqref{norm-relation}, we have after some direct manipulations
	\begin{equation}
	\begin{split}
	&\Big\|\int_0^{c_i\Delta t }e^{-(c_i\Delta t -\tau)L_h}\int_0^{\tau}\frac{(\tau-\sigma)^{r-1}}{(r-1)!}P_hf^{(r)}(t_n+\sigma,u(t_n+\sigma))\dd \sigma \dd \tau\Big\|_1
	\\
	&\;\lesssim \Big\|\int_0^{c_i\Delta t }e^{-(c_i\Delta t -\tau)L_h}\int_0^{\tau}\frac{(\tau-\sigma)^{r-1}}{(r-1)!}L_h^{\frac 1 2}P_hf^{(r)}(t_n+\sigma,u(t_n+\sigma))\dd\sigma\dd\tau\Big\|_0 
	\\
	&\;\lesssim  \textstyle(\Delta t)^{r+1}\sup\limits_{0 \leq \tau \leq c_i\Delta t }\|e^{-(c_i\Delta t -\tau)L_h}\|_0 \textstyle\sup\limits_{0\leq \eta \leq 1}\|L_h^{\frac 1 2} P_hf^{(r)}(t_n+\eta \Delta t ,u(t_n+\eta \Delta t ))\|_0\\
	&\;\lesssim (\Delta t) ^{r+1}\textstyle\sup\limits_{0\leq \eta \leq 1}\|P_hf^{(r)}(t_n+\eta \Delta t ,u(t_n+\eta \Delta t ))\|_1\\
	&\;\lesssim (\Delta t) ^{r+1}\textstyle\sup\limits_{0\leq \eta \leq 1}\|f^{(r)}(t_n+\eta \Delta t ,u(t_n+\eta \Delta t ))\|_1, 
	\end{split}
	\label{mid-estimate1}
	\end{equation}
	where  the last inequality is due to that $P_h$ is $H^1$-stable.
	Similarly, we also have
	\begin{equation}
	\begin{split}
	&\Big\| \int_0^{\Delta t }e^{-(\Delta t -\sigma)L_h}\Big(P_hf(t_n+\sigma,u_h(t_n+\sigma))-P_hf(t_n+\sigma,u(t_n+\sigma)) \Big)\dd \sigma\Big\|_1 
	\\
	&\lesssim \Big\|\int_0^{\Delta t }L_h^{\frac 1 2}e^{-(\Delta t -\sigma)L_h}\dd \sigma\Big\|_0 \textstyle\sup\limits_{0 \leq \eta \leq 1} \|P_hf(t_n+\eta \Delta t ,u_h(t_n+\eta \Delta t )) 
	\\
	&\quad -P_hf(t_n+\eta\Delta t ,u(t_n+\eta\Delta t ))\Big\|_0. 
	\end{split}
	\label{eq4-3}
	\end{equation}
	
	In view of the similar arguments for proving Lemma \ref{lemma1}, we have
	\begin{equation}
	\begin{split}
	\Big\|\int_0^{\Delta t }L_h^{\frac 1 2}e^{-(\Delta t -\sigma)L_h}\dd \sigma \Big\|_0 &= \textstyle\max\limits_{\lambda \in \lambda(L_h)} \Big|\int_0^{\Delta t }\lambda^{\frac 1 2}e^{-(\Delta t -\sigma)\lambda}\dd \sigma \Big| \\
	&\leq\textstyle\max\limits_{\lambda \in \lambda(L_h)} |\lambda^{-\frac 1 2}|\lesssim 1. 
	\end{split}
	\label{eq4-4a}
	\end{equation}
	Since $P_h$ is $L^2$ stable and the function $f$ is locally-Lipschitz continuous (Lemma \ref{lemma3}), we further obtain from \eqref{eq4-3}-\eqref{eq4-4a} and Lemma \ref{semi-estimate}  that
	\begin{equation}
	\begin{split}
	&\Big\| \int_0^{\Delta t }e^{-(\Delta t -\sigma)L_h}\big(P_hf(t_n+\sigma,u_h(t_n+\sigma))-P_hf(t_n+\sigma,u(t_n+\sigma)) \big)\dd \sigma\Big\|_1 
	\\
	&\lesssim \textstyle\sup\limits_{0 \leq \eta \leq 1} \|P_hf(t_n+\eta \Delta t ,u_h(t_n+\eta \Delta t ))-P_hf(t_n+\eta \Delta t ,u(t_n+\eta \Delta t ))\|_0 
	\\
	&\lesssim  \textstyle\sup\limits_{0 \leq \eta \leq 1}\|u_h(t_n+\eta \Delta t )-u(t_n+\eta \Delta t )\|_1\lesssim h. 
	\end{split}
	\label{mid-estimate2}
	\end{equation}
	According to \eqref{666} in Lemma \ref{lemma1} and the similar arguments for deriving \eqref{mid-estimate1}, 
	\begin{equation}
	\begin{split}
	\quad&\Big\|\Delta t \textstyle\sum\limits_{k=1}^{i-1} a_{ik}(-\Delta t L_h)\displaystyle\int_0^{c_k\Delta t }\frac{(c_k\Delta t -\sigma)^{r-1}}{(r-1)!}P_hf^{(r)}(t_n+\sigma,u(t_n+\sigma))\dd \sigma\Big\|_1 
	\\
	&\lesssim \Big\|\Delta t \textstyle\sum\limits_{k=1}^{i-1} a_{ik}(-\Delta t L_h)\displaystyle\int_0^{c_k\Delta t }\frac{(c_k\Delta t -\sigma)^{r-1}}{(r-1)!}L_h^{\frac 1 2}\\
	&\qquad\quad \cdot P_hf^{(r)}(t_n+\sigma,u(t_n+\sigma))\dd \sigma\Big\|_0 
	\\
	&\leq  (\Delta t)^{r+1} \textstyle\sum\limits_{k=1}^{i-1}\|a_{ik}(-\Delta t L_h)\|_0\sup\limits_{0 \leq \eta \leq 1}\|P_hf^{(r)}(t_n+\eta \Delta t ,u(t_n+\eta \Delta t ))\|_1 \\
	&\lesssim (\Delta t)^{r+1}\textstyle\sup\limits_{0 \leq \eta \leq 1}\|f^{(r)}(t_n+\eta \Delta t ,u(t_n+\eta \Delta t ))\|_1. 
	\end{split}
	\label{mid-estimate3}
	\end{equation}
	
	Now, using the triangle inequality, the regularity assumptions for $u(t)$ and $f$, and the estimates \eqref{mid-estimate1}, \eqref{mid-estimate2} and \eqref{mid-estimate3}, we get
	\begin{equation}
	\|\delta_{ni}^{[r]}\|_1 \lesssim (\Delta t )^{r+1}\|f^{(r)}(t_n+\eta \Delta t ,u(t_n+\eta \Delta t ))\|_1+ h,\quad\forall\, i=1, \cdots, s,
	\end{equation}
	which leads to \eqref{lemma4-1}.  Also \eqref{lemma4-2} can be derived in the similar manner.
\end{proof}

For the EIFE1 scheme \eqref{Euler} (i.e., $s=1$),  the following error equation  holds:
\begin{equation}\label{errequ1}
e_{n+1} = e^{-\Delta t L_h}e_n+\Delta t \varphi_1(-\Delta t L_h)\big(P_hf(t_n,u_h^n)-P_hf(t_n,u(t_n)) \big) - \delta_{n+1}.
\end{equation}

\begin{theorem}[Error estimate for the EIFE1 scheme]
	\label{thm7}
	Suppose the function $f$ satisfies Assumptions \ref{assumption2} and \ref{assumption3}, and the exact solution $u(t)$ fulfills (\ref{strong-regularity2}) and (\ref{strong-regularity1}). There exists a constant $h_0>0$ such that if the spatial mesh size $h \leq h_0$, then the numerical solution $\{u_h^n\}$ produced by the EIFE1 scheme  (\ref{Euler})  satisfies
	\begin{equation}\label{err1}
	\|u(t_n)-u_h^n\|_1 \lesssim \Delta t + h, \quad \forall\,n=1,\cdots, N_T,
	\end{equation}
	where the hidden constant is independent of $h$ and $\Delta t$.
\end{theorem}

\begin{proof}
	By repeatedly applying \eqref{errequ1}, we get
	\begin{equation}
	\label{Euler-error}
	\begin{split}
	e_n = \;&\Delta t \textstyle\sum\limits_{j=1}^{n-1} e^{-(n-j-1)\Delta t L_h}\varphi_1(-\Delta t L_h)\big(P_hf(t_j,u_h^j)-P_hf(t_j,u(t_j)) \big)\\
	&-\textstyle\sum\limits_{j=0}^{n-1}e^{-j\Delta t L_h}\delta_{n-j}.
	\end{split}
	\end{equation}
	
	Since $s=1$, based on the definition of $\psi_i(-\Delta t L_h)$ in (\ref{function}) and Lagrangian interpolation theorem, we can derive that $\psi_1(-\Delta t L_h)=0$, which implies  $\delta_j=\delta_j^{[1]}$ for $j=1, \cdots, n$. Recalling \eqref{lemma4-2} in Lemma \ref{lemma4}  (with $r=1$), we can immediately obtain the $H^1$-norm estimation of the second term in \eqref{Euler-error} as
	\begin{align}
	\label{eq4-6a}
	\Big\|\textstyle\sum\limits_{j=0}^{n-1}e^{-j\Delta t L_h}\delta_{n-j}\Big\|_1 &= \Big\|\textstyle\sum\limits_{j=0}^{n-1}e^{-j\Delta t L_h}\delta_{n-j}^{[1]}\Big\|_1 \lesssim \Delta t\textstyle\sup\limits_{ 0  \leq t \leq    T }\|f^{(1)}(t,u(t))\|_1+h.
	\end{align} 
	As for the $H^1$-norm evaluation of the first term in \eqref{Euler-error}, it follows from \eqref{norm-relation} and the triangle inequality that {
		\begin{equation}
		\label{eq4-6}
		\begin{split}
		&\Big\|\Delta t \textstyle\sum\limits_{j=1}^{n-1}e^{-(n-j-1)\Delta t L_h}\varphi_1(-\Delta t L_h)\big(P_hf(t_j,u_h^j)-P_hf(t_j,u(t_j)) \big)\Big\|_1\\
		&\lesssim  \|\Delta t L_h^{\frac 1 2}\varphi_1(-\Delta t L_h)\|_0\|P_hf(t_{n-1},u_h^{n-1})-P_hf(t_{n-1},u(t_{n-1}))\|_0\\
		&\quad+\Big\|\textstyle\sum\limits_{j=0}^{n-2}\Delta t L_h^{\frac 1 2}e^{-(n-j-1)\Delta t L_h}\varphi_1(-\Delta t L_h)\big(P_hf(t_j,u_h^j)-P_hf(t_j,u(t_j)) \big)\Big\|\\
		&:={\rm I}_1+{\rm I}_2.
		\end{split}
		\end{equation}}
	By applying \eqref{eq1-8} in Lemma \ref{lemma1}, we get
	\begin{equation}
	\label{eq4-7}
	\begin{split}
	{\rm I}_1 =\,& \Big\|L_h^{\frac 1 2}\int_0^{\Delta t }e^{-(\Delta t -\sigma)L_h}\dd \sigma\Big\|_0\|P_hf(t_{n-1},u_h^{n-1})-P_hf(t_{n-1},u(t_{n-1}))\|_0\\
	\lesssim\,& \Delta t \textstyle\sup\limits_{0 \leq \sigma \leq \Delta t }\|L_h^{\frac 1 2}e^{-(\Delta t -\sigma)L_h}\|_0\|P_hf(t_{n-1},u_h^{n-1})-P_hf(t_{n-1},u(t_{n-1}))\|_0\\
	\lesssim\,& (\Delta t)^{\frac 1 2}\|u_h^{n-1}-u(t_{n-1})\|_1.
	\end{split}
	\end{equation}
	For the second term in \eqref{eq4-6}, it follows from  \eqref{eq1-8} in Lemma \ref{lemma1} that
	\begin{equation*}
	\|\varphi_1(-\Delta t L_h)\|_0 = \Big\|\frac{1}{\Delta t }\int_0^{\Delta t }e^{-(\Delta t -\tau)L_h}\dd \tau\Big\|_0 \leq \textstyle\sup\limits_{0 \leq \tau \leq \Delta t }\|e^{-(\Delta t -\tau)}\|_0 \lesssim 1.
	\end{equation*}
	By \eqref{eq1-8} and \eqref{555} in Lemma \ref{lemma1} and similar arguments for deriving \eqref{mid-estimate2}, we  have {
		\begin{equation}
		\label{eq4-7a}
		\begin{split}
		{\rm I}_2\lesssim\, &\Big\|\Delta t L_h^{\frac 1 2}\textstyle\sum\limits_{j=0}^{n-2}e^{-(n-1-j)\Delta t L_h}\Big\|_0\textstyle\sup\limits_{ 0  \leq t \leq    T }\|P_hf(t,u_h(t))-P_hf(t,u(t))\|_0\\
		&+\Delta t \textstyle\sum\limits_{j=0}^{n-2}\|L_h^{\frac 1 2}e^{-(n-1-j)\Delta t L_h}\|_0\|P_hf(t_j,u_h^j)-P_hf(t_j,u_h(t_j))\|_0\\
		\lesssim \,& \textstyle\sup\limits_{ 0  \leq t \leq    T }\|P_hf(t,u_h(t))-P_hf(t,u(t))\|_0+\Delta t \textstyle\sum\limits_{j=0}^{n-2}t_{n-j-1}^{-\frac 1 2}\|P_hf(t_j,u_h^j)\\
		&\qquad-P_hf(t_j,u_h(t_j))\|_0\\
		\lesssim\, & \Delta t \textstyle\sum\limits_{j=0}^{n-2}t_{n-j-1}^{-\frac 1 2}\|e_j\|_1+h.
		\end{split}
		\end{equation}}
	In terms of \eqref{Euler-error} and the estimates \eqref{eq4-6a}-\eqref{eq4-7a}, we arrive at 
	\begin{equation*}
	\begin{split}
	\|e_n\|_1 &\lesssim (\Delta t)^{\frac 1 2}\|e_{n-1}\|_1 + \Delta t  \textstyle\sum\limits_{j=0}^{n-2}t_{n-j-1}^{-\frac 1 2}\|e_j\|_1 + \Delta t + h\\
	&\lesssim \Delta t \textstyle\sum\limits_{j=1}^{n-1}t_{n-j}^{-\frac 1 2}\|e_j\|_1+\Delta t+h.
	\end{split}
	\end{equation*}
	Then we have by the discrete Gronwall inequality (Theorem 6.1 with $\alpha=\frac 1 2$ in \cite{DixonMcKee1986})
	\begin{equation}\label{err1c}
	\|u_h(t_n)-u_h^n\|_1 \lesssim \Delta t+ h.
	\end{equation}
	The combination of \eqref{errtri}, \eqref{erra} and \eqref{err1c} immediately leads to \eqref{err1}.
\end{proof}

\begin{lemma}
	\label{Lemma5}
	Suppose the function $f$ satisfies Assumptions \ref{assumption2} and \ref{assumption3}, and  the exact solution $u(t)$ fulfills (\ref{strong-regularity2})-(\ref{strong-regularity1}). If $s\geq 2$, then it holds for any  $0 \leq n < N_T$,
	\begin{equation}
	\label{E_n}
	\|E_{ni}\|_1 \lesssim \|e_n\|_1 + (\Delta t)^2 \textstyle\sup\limits_{0\leq \eta \leq 1} \|f'(t_n+\eta \Delta t ,u(t_n+\eta \Delta t ))\|_1 + h, \quad \forall\, i=1, \cdots, s,
	\end{equation}
	where the hidden constant is independent of $h$ and $\Delta t$.
\end{lemma}

\begin{proof}
	According to the definition of $\psi_{j,i}$ in (\ref{function}), we have by some manipulations that $\psi_{1, j}=0, j=1, \cdots, s$ when the consistency conditions (\ref{eq2-3}) is fulfilled. Therefore, the estimation of $\|\delta_{ni}\|_1$ can be converted to that of $\|\delta_{ni}^{[1]}\|_1$. Using the similar arguments for deriving the estimate \eqref{mid-estimate2} and \eqref{666} in Lemma \ref{lemma1}, we have 
	\begin{align}
	\begin{split} &\Big\|\Delta t \textstyle\sum\limits_{j=1}^{i-1}a_{ij}(-\Delta t L_h)\big(P_hf(t_n+c_j\Delta t ,U_{nj})-P_hf(t_n+c_j\Delta t ,u(t_n+c_j\Delta t )) \big)\Big\|_1\\
	&\lesssim  \textstyle\sum\limits_{j=1}^{i-1}(\Delta t)^{\frac12}\|\Delta t ^{\frac 1 2}L_h^{\frac 1 2}a_{ij}(-\Delta t L_h)\|_0\nonumber\\
	&\qquad \cdot\textstyle\max\limits_{2 \leq j \leq i -1} \|P_hf(t_n+c_j\Delta t ,U_{nj})-P_hf(t_n+c_j\Delta t ,u(t_n+c_j\Delta t ))\|_0\\
	&\lesssim (\Delta t)^{\frac 1 2}\textstyle\max\limits_{2 \leq j \leq i -1} \|P_hf(t_n+c_j\Delta t ,U_{nj})-P_hf(t_n+c_j\Delta t ,u(t_n+c_j\Delta t ))\|_0\\
	&\lesssim  (\Delta t)^{\frac 1 2}\textstyle\max\limits_{2 \leq j \leq i-1}\|E_{nj}\|_1+h.
	\end{split}
	\end{align}
	Note that $\|\delta_{ni}^{[1]}\|_1$ is uniformly bounded for $i = 1, \cdots, s$ (see \eqref{lemma4-1} in Lemma \ref{lemma4} with $r=1$). Recalling the relation \eqref{recursion-E}, we have by the triangle inequality that {
		\begin{equation*}
		\begin{split}
		\|E_{ni}\|_1 \lesssim&\; \|e^{-c_i\Delta t L_h}e_n\|_1+(\Delta t)^2\textstyle\sup\limits_{0 \leq \eta \leq 1}\|P_hf'(t_n+\eta\Delta t ,u(t_n+\eta \Delta t ))\|_1+h\\
		&+\Big\|\Delta t  \textstyle\sum\limits_{j=1}^{i-1}a_{ij}(-\Delta t L_h)\big(P_hf(t_n+c_j\Delta t ,U_{nj})\\
		&\qquad-P_hf(t_n+c_j\Delta t ,u(t_n+c_j\Delta t )) \big)\Big\|_1\\
		\lesssim&\; \|e_n\|_1 +  (\Delta t)^2 \textstyle\sup\limits_{0\leq \eta \leq 1} \|P_hf'(t_n+\eta \Delta t , u(t_n+\eta \Delta t ))\|_1+ h\\
		&+ (\Delta t)^{\frac 1 2}\textstyle\max\limits_{2 \leq j \leq i -1}\|E_{nj}\|_0\\
		\lesssim &\; \|e_n\|_1 +  (\Delta t)^2 \textstyle\sup\limits_{0\leq \eta \leq 1}\|f'(t_n+\eta \Delta t , u(t_n+\eta \Delta t ))\|_1 + h \\
		&+ (\Delta t)^{\frac 1 2}\textstyle\max\limits_{2 \leq j \leq i -1} \|E_{nj}\|_1.
		\end{split}
		\end{equation*}}
	Finally \eqref{E_n} is obtained by recursively using the above inequality.
\end{proof}

\begin{theorem}[Error estimate for the EIFE2 scheme] \label{thm4}
	Suppose that the function $f$ satisfies Assumptions \ref{assumption2} and \ref{assumption3}, and the exact solution $u(t)$ fulfills (\ref{strong-regularity2})-(\ref{strong-regularity3}) in Assumptions \ref{assumption4}.   There exists a constant $h_0>0$ such that if the spatial mesh size $h \leq h_0$, then the numerical solution $\{u_h^n\}$ produced by EIFE2 scheme  \eqref{second-order} satisfies
	\begin{equation}\label{err2}
	\|u(t_n)-u_h^n\|_1 \lesssim (\Delta t)^2 + h,\quad \forall\,n=1,\cdots, N_T,
	\end{equation}
	where the hidden constant is independent of $h$ and $\Delta t$.	
\end{theorem}

\begin{proof}	
	Recalling the definition of $\psi_i$ in (\ref{function}), we can check that $\psi_{1}(-\Delta t L_h)=\psi_2(-\Delta t L_h)=0$, which implies that $\delta_{n+1}=\delta_{n+1}^{[2]}$ by Lagrangian interpolation theorem for $s = 2$. By \eqref{666} in Lemma \ref{lemma1}, we have
	\begin{equation}
	\label{eq4-4}
	\begin{split}
	&\Big\|\Delta t \textstyle\sum\limits_{j=0}^{n-1} e^{-(n-1-j)\Delta t L_h}\textstyle\sum\limits_{i=1}^s b_i(-\Delta t L_h)\big(P_hf(t_j+c_i\Delta t , U_{ji})\\
	&	\qquad-P_hf(t_j+c_i\Delta t ,u(t_j+c_i\Delta t )) \big)\Big\|_1
	\\
	&\lesssim  \Big\|\Delta t \textstyle\sum\limits_{i=1}^s b_i(-\Delta t L_h)\big(P_hf(t_{n-1}+c_i\Delta t ,U_{n-1,i})\\
	&\qquad-P_hf(t_{n-1}+c_i\Delta t ,u(t_{n-1}+c_i\Delta t ) \big)\Big\|_1 
	\\
	&\qquad+ \Big\|\Delta t \textstyle\sum\limits_{j=0}^{n-2} e^{-(n-1-j)\Delta t L_h}\textstyle\sum\limits_{i=1}^s \big(P_hf(t_j+c_i\Delta t ,U_{ji})\\
	&\qquad-P_hf(t_j+c_i\Delta t ,u(t_j+c_i\Delta t )) \big)\Big\|_1\\
	&:= {\rm II}_1+{\rm II}_2.
	\end{split}
	\end{equation}
	Using the similar arguments for deriving \eqref{mid-estimate2} and \eqref{555} and \eqref{666} in Lemma \ref{lemma1},  we can obtain
	\begin{equation}
	\label{add-1}
	\begin{split}
	{\rm II}_1\lesssim \,& \textstyle\sum\limits_{i=1}^s (\Delta t) ^{\frac 1 2}\|(\Delta t) ^{\frac 1 2}L_h^{\frac 1 2}b_i(-\Delta t L_h)\|_0\textstyle\max\limits_{1\le i\le s}\|P_hf(t_{n-1}+c_i\Delta t ,U_{n-1,i})\\
	&\qquad\qquad-P_hf(t_{n-1}+c_i\Delta t ,u(t_{n-1}+c_i\Delta t ))\|_0\\
	\lesssim \,& (\Delta t)^{\frac 1 2}\textstyle\max\limits_{1\le i\le s} \|P_hf(t_{n-1}+c_i\Delta t ,U_{n-1,i})\\
	&\qquad\qquad-P_hf(t_{n-1}+c_i\Delta t ,u(t_{n-1}+c_i\Delta t ))\|_0 \\
	\lesssim \,& (\Delta t)^{\frac 1 2}\textstyle\max\limits_{1 \leq i \leq s} \|E_{n-1,i}\|_1+h,
	\end{split}
	\end{equation}
	and
	\begin{equation}
	\label{add-2}
	\begin{split}
	{\rm II}_2
	\lesssim \,& \Big\|\Delta t L_h^{\frac 1 2}\textstyle\sum\limits_{j=0}^{n-2}e^{-(n-1-j)\Delta t L_h}\Big\|_0\textstyle\sup\limits_{ 0  \leq t \leq    T }\|P_hf(t,u(t))-P_hf(t,u_h(t))\|_0 \\&+ \textstyle\sum\limits_{j=0}^{n-2} \Delta t \|L_h^{\frac 1 2} e^{-(n-j-1)\Delta t L_h}\|_0\textstyle\max\limits_{1\le i \le s} \|P_hf(t_j+c_i\Delta t ,U_{ji})\\
	&\qquad-P_hf(t_j+c_i\Delta t ,u_h(t_j+c_i\Delta t ))\|_0 
	\\
	\lesssim\,& \Delta t\textstyle\sum\limits_{j=0}^{n-2}t_{n-j-1}^{-\frac 1 2}\textstyle\max\limits_{1\le i \le s}\|P_hf(t_j+c_i\Delta t ,U_{ji})\\
	&\qquad-P_hf(t_j+c_i\Delta t ,u_h(t_j+c_i\Delta t ))\|_0+h 
	\\
	\lesssim \,& \Delta t\textstyle\sum\limits_{j=0}^{n-2}t_{n-j-1}^{-\frac 1 2}\textstyle\max\limits_{1 \leq i \leq s} \|E_{ji}\|_1+h.
	\end{split}
	\end{equation}
	
	With the help of the estimates \eqref{eq4-4}, \eqref{add-1} and \eqref{add-2}, it follows from the relation \eqref{recursion-e} and \eqref{lemma4-2} in Lemma \ref{lemma4} (with $r=2$) that
	
	\begin{equation*}
	\begin{split}
	\|e_n\|_1 \leq\,& \Big\|\Delta t \textstyle\sum\limits_{j=0}^{n-1} e^{-(n-1-j)\Delta t L_h}\textstyle\sum\limits_{i=1}^s b_i(-\Delta t L_h)\big(P_hf(t_j+c_i\Delta t , U_{ji})\nonumber\\
	&\;-P_hf(t_j+c_i\Delta t ,u(t_j+c_i\Delta t )) \big)\Big\|_1+\Big\|\textstyle\sum\limits_{j=0}^{n-1}e^{-j\Delta t L_h}\delta_{n-j}^{[2]}\Big\|_1
	\\
	\lesssim \,& (\Delta t)^{\frac 1 2}\textstyle\max\limits_{1\le i \le s}  \|E_{n-1,i}\|_1
	+\Delta t\textstyle\sum\limits_{j=0}^{n-2}t_{n-j-1}^{-\frac 1 2}\textstyle\max\limits_{1\le i \le s}\|E_{ji}\|_1\\
	&
	+\Delta t^2\textstyle\sup\limits_{ 0  \leq t \leq    T }\|f^{(2)}(t,u(t))\|_1+h
	\\
	\lesssim \,& \Delta t\textstyle\sum\limits_{j=0}^{n-1}t_{n-j-1}^{-\frac 1 2}\textstyle\max\limits_{1\le i \le s} \|E_{ji}\|_1+(\Delta t)^2+h.
	\end{split}
	\end{equation*}
	This combined with the estimation of $\|E_{j,i}\|_1$ in Lemma \ref{Lemma5} and the discrete Gronwall inequality leads to
	\begin{equation}\label{err2c}
	\|u_h^n-u_h(t_n)\|_1 \lesssim (\Delta t)^2 + h.
	\end{equation}
	Finally, the combination of \eqref{errtri}, \eqref{erra} and \eqref{err2c} immediately gives  \eqref{err2}.	
\end{proof}

\begin{remark}
	It is worth noting that  no restriction on the time step size $\Delta t$ is imposed in Theorems \ref{thm7} and \ref{thm4}, which implies that the proposed EIFE method is stable with large time stepping.
\end{remark}

\begin{remark}
	When $s\ge 3$, the order conditions for the explicit exponential Runge-Kutta method do not hold automatically like the case of $s=1$ or $2$, and
	convergence analysis of the proposed EIFE method \eqref{fully} becomes much more complicated since the estimates of $\|e_n\|_1$ and $\|E_{ni}\|_1$ will be coupled together. We refer the reader to \cite{HochbruckOstermann2005b} for some details along this line, and rigorous error estimates of
	higher-order EIFE schemes would be an interesting open question.
	
\end{remark}

\section{FFT-based fast solution of the EIFE method}
\label{fastimpl}
We now present  fast solution of the EIFE method \eqref{fully} for solving the semilinear parabolic equation \eqref{eq1-1}.
The main idea is to take the advantage of simultaneously diagonalizing  the mass and coefficient matrices of the finite element system \eqref{eq3-11} with an orthogonal matrix whose multiplication with a vector can be efficiently implemented by FFT and tensor product spectral decomposition.


We  first focus on the case of homogeneous Dirichlet boundary condition \eqref{eq3-1}, and the extension to that of  nonhomogeneous Dirichlet boundary condition is straightforward by incorporating the given values of the exact solution on the boundary nodes into the discrete method as part of $f$.
Let us take the three dimensional problem for illustration, which also naturally works in any other $d$ dimensions. Suppose that $\Omega := \prod_{i=1}^3 [a_i,b_i]$ is divided into $N_x,N_y,N_z$ parts uniformly along the $x, y, z$ directions with the meshsizes $h_x, h_y, h_z$, respectively. Thus, as given in Section \ref{algorithm-description}, the finite element space $V_h$ is spanned by $\{\phi_{i}(x)\phi_j(y)\phi_k(z) \}$ where $i = 1, \cdots, N_x-1, j=1, \cdots, N_y-1, k = 1, \cdots, N_z-1$. Therefore, the finite element solution $u_h(t)$ of  the
semi-discrete (in space) problem \eqref{eq3-11} can be expressed as
\begin{equation*}
u_h(t, x, y, z) = \textstyle\sum\limits_{i=1}^{N_x-1} \textstyle\sum\limits_{j=1}^{N_y-1} \textstyle\sum\limits_{k = 1}^{N_z-1} u_{i,j,k}(t) \phi_i(x)\phi_j(y)\phi_k(z).
\end{equation*}
Define $U(t)=[u_{i, i, k}(t)]_{(N_x-1)\times (N_y-1) \times (N_z-1)}$
and let $F(t,U)$ be  the tensor with entries
$(f(t,u_{i, j, k}(t)), \phi_{i}(x)\phi_{j}(y)\phi_{k}(z)).$ As in \cite{JuZhang2015}, we introduce some tensor operations for later uses. Given  three matrices  $M_x\in{\mathbb R}^{(N_x-1)\times (N_x-1)}$, $M_y\in{\mathbb R}^{(N_y-1)\times (N_y-1)}$, $M_z\in{\mathbb R}^{(N_z-1)\times (N_z-1)}$, let us  define the following operations with $U$:
$\big(M_x \textcircled{x} U \big)_{ijk} := \textstyle\sum\limits_{r=1}^{N_x-1} (M_x)_{ir}U_{rjk}$, $\big(M_y \textcircled{y} U \big)_{ijk}:= \textstyle\sum\limits_{r=1}^{N_y-1} (M_y)_{jr}U_{irk}$, $\big(M_z \textcircled{z} U \big)_{ijk}:=\textstyle\sum\limits_{r=1}^{N_z-1} (M_z)_{kr}U_{ijr}$.

With the help of the above operations, we have by some direct manipulations that the variational scheme \eqref{eq3-11} can be reformulated to the following form
\begin{equation}
\label{eq3-3}
\left\{
\begin{split}
&A_x \textcircled{x} A_y \textcircled{y} A_z \textcircled{z} \frac{\dd{U}}{\dd t}+ D(B_x \textcircled{x} A_y \textcircled{y} A_z \textcircled{z} U \\
&\qquad\qquad\qquad+ A_x \textcircled{x} B_y \textcircled{y} A_z \textcircled{z} U  + A_x \textcircled{x} A_y \textcircled{y} B_z \textcircled{z} U)  = F(t,U),
\\
&U( 0 ) = \{P_h u_0\},
\end{split}
\right.
\end{equation}
where
\begin{equation*}
A_{\Diamond} = \textstyle\frac{h_{\Diamond}}{6} R_{N_{\Diamond}-1},  \quad B_{\Diamond}  =\textstyle \frac{1}{h_{\Diamond}} G^D_{N_{\Diamond} -1},\quad\mbox{for}\;\;{\Diamond}=x,y,z,
\end{equation*}
and 
\begin{equation*}
\begin{split}
R_p = \left[
\begin{matrix}
4 & 1  & &\\
1 & 4 & \ddots & \\
& \ddots & \ddots & 1\\
& & 1 & 4
\end{matrix}
\right]_{p \times p}, \quad
G^D_p = \left[
\begin{matrix}
2 & -1  & &\\
-1 & 2 & \ddots & \\
& \ddots & \ddots & -1\\
& & -1 & 2
\end{matrix}
\right]_{p \times p}.
\end{split}
\end{equation*}

An important observation is that   $A_x$ and $B_x$ can be diagonalized simultaneously by an orthogonal matrix  (\cite{HuangJu2019b}), and so do  the other two pairs $A_y$ and $B_y$, $A_z$ and $B_z$. Specifically, we have the following decompositions:
$$A_{\Diamond}=P_{\Diamond}\Lambda_{A_{\Diamond}}P_{\Diamond}^T,\quad B_{\Diamond}=P_{\Diamond}\Lambda_{B_{\Diamond}}P_{\Diamond}^T, \quad\mbox{for}\;\;{\Diamond}=x,y,z$$
with
$$ (P_{\Diamond})_{i,j} = \sin\Big(\textstyle\frac{ij\pi}{N_{\Diamond}} \Big), \quad i,j=1, \cdots, N_{\Diamond}-1,$$
$$\Lambda_{A_{\Diamond}} = \text{diag}(\lambda^{A_{\Diamond}}_1,\cdots,\lambda^{A_{\Diamond}}_{N_{\Diamond}-1}),\quad\lambda^{A_{\Diamond}}_i =\textstyle \frac {h_{\Diamond}}{6} \Big(6-4\sin\Big(\frac{i\pi}{2N_{\Diamond}}\Big)\Big),\quad i=1, \cdots, N_{\Diamond}-1,$$
$$\Lambda_{B_{\Diamond}} = \text{diag}(\lambda^{B_{\Diamond}}_1,\cdots,\lambda^{B_{\Diamond}}_{N_{\Diamond}-1}),\quad\lambda^{B_{\Diamond}}_i = \textstyle\frac {4}{h_{\Diamond}} \Big(\sin^2\Big(\frac{i\pi}{2N_{\Diamond}}\Big) \Big),\quad i=1, \cdots, N_{\Diamond}-1.$$

Let us define $ \widetilde{U}= P_x^T \textcircled{x} P_y^T \textcircled{y} P_z^T \textcircled{z} U$ (i.e., $U= P_x \textcircled{x} P_y\textcircled{y} P_z \textcircled{z}  \widetilde{U}$). Then  the equation (\ref{eq3-3}) can be transformed to
\begin{equation}
\label{append2}
\begin{split}
&\Lambda_{A_x} \textcircled{x} \Lambda_{A_y} \textcircled{y} \Lambda_{A_z} \textcircled{z} \frac{\dd\widetilde{U}}{\dd t} + D\Big(\Lambda_{B_x} \textcircled{x} \Lambda_{A_y} \textcircled{y} \Lambda_{A_z} \textcircled{z} + \Lambda_{A_x} \textcircled{x} \Lambda_{B_y} \textcircled{y} \Lambda_{A_z} \textcircled{z}  \\&
\qquad\quad+ \Lambda_{A_x} \textcircled{x} \Lambda_{A_y} \textcircled{y} \Lambda_{B_z} \textcircled{z}\Big) \widetilde{U}
= P_x^T \textcircled{x} P_y^T \textcircled{y} P_z^T \textcircled{z} F(t,P_x \textcircled{x} P_y\textcircled{y} P_z \textcircled{z}  \widetilde{U}).
\end{split}
\end{equation}
Similar to \cite{ZhuJu2016},  we define the tensor  $H=(h_{ijk})_{(N_x-1)\times (N_y-1) \times (N_z-1)}$ with
$$h_{ijk}=D\Big(\textstyle\frac{(\Lambda_{B_x})_{ii}}{(\Lambda_{A_x})_{ii}}+\frac{(\Lambda_{B_y})_{jj}}{(\Lambda_{A_y})_{jj}}+\frac{(\Lambda_{B_z})_{kk}}{(\Lambda_{A_z})_{kk}} \Big)$$
and another tensor $\widehat{H}=(\widehat{h}_{ijk})_{(N_x-1)\times (N_y-1) \times (N_z-1)}$ with $$\hat{h}_{ijk}=\textstyle\frac{1}{(\Lambda_{A_x})_{ii}(\Lambda_{A_y})_{jj}(\Lambda_{A_z})_{kk}}.$$
Define the operator $(e^*)$ as taking the exponential of each entry of a tensor as
$((e^*)^H)_{ijk}=e^{h_{ijk}},$
and another operator $\odot$ for element by element multiplication between two arrays of same sizes as
$(A \odot B)_{ijk}=A_{ijk}B_{ijk}.$
Thus we can rewrite the system (\ref{append2}) to the following ODE system:
\begin{equation}
\label{append3}
\frac{\dd\widetilde{U}}{\dd t} + H\odot\widetilde{U} = \widehat{H}\odot\big(P_x^T \textcircled{x} P_y^T  \textcircled{y} P_z^T \textcircled{z} F(t,P_x \textcircled{x} P_y\textcircled{y} P_z \textcircled{z}  \widetilde{U})\big).
\end{equation}
Hence, by the Duhamel principle, we can get the following  expression
\begin{equation}
\label{append4}
\begin{split}
&\widetilde{U}(t_{n+1}) = (e^*)^{-H\Delta t }\odot\widetilde{U}(t_n) + \int_{t_n}^{t_{n+1}} (e^*)^{-H(t_{n+1}-\tau)}\odot\widehat{H}\\
&\qquad\qquad\qquad\qquad\odot\big(P_x^T \textcircled{x} P_y^T  \textcircled{y} P_z^T \textcircled{z} F (\tau,P_x \textcircled{x} P_y\textcircled{y} P_z \textcircled{z}  \widetilde{U}) \big)\dd \tau,
\end{split}
\end{equation}
which is equivalently  the tensor representation of  \eqref{Duhamel}.

We remark that all the tensor product operations can be realized by the FFT. More precisely, the above transformation from $U$ to $\widetilde{U}$ (resp. from $\widetilde{U}$ to $U$) can be implemented by applying a Discrete Sine Transform (DST) (resp. an inverse DST, abbreviated as iDST) in each  dimension separately. That means,  for a $d$-dimensional problem, the computational  cost of the  EIFE method \eqref{fully} is of $O(\log_2(N)\prod_{i=1}^d N_i)$ per time step, where $N=\textstyle\max\{N_1,\cdots,N_d \}$. 

Next, we  also would like to briefly discuss  the equation \eqref{eq1-1} with periodic  boundary condition (although not theoretically analyzed in this paper). In this case,
the finite element space $V_h$ is spanned by $\{\phi_{i}(x)\phi_j(y)\phi_k(z) \}$ where $i = 1, \cdots, N_x, j=1, \cdots, N_y, k = 1, \cdots, N_z$ by periodic extension.
The same EIFE method and efficient implementation algorithm can be directly applied with slight modifications. In fact, the unknowns in this case is a three-order tensor given by
$U(t)=[u_{i, j, k}(t)]_{N_x\times N_y \times N_z}$,
and the corresponding mass  and coefficient matrices become respectively
\begin{equation*}
\begin{split}
A_{\Diamond} =   \textstyle\frac{h_{\Diamond}}{6} R_{N_{\Diamond}},  \quad B_{\Diamond}  = \textstyle\frac{1}{h_{\Diamond}} G^P_{N_{\Diamond}},\quad\mbox{for}\;\;{\Diamond}=x,y,z,
\end{split}
\end{equation*}
where
\begin{equation*}
G^P_p = \left[
\begin{matrix}
2 & -1  & &-1\\
-1 & 2 & \ddots & \\
& \ddots & \ddots & -1\\
-1& & -1 & 2
\end{matrix}
\right]_{p \times p}.
\end{equation*}
Then we have
$$A_{\Diamond}=P_{\Diamond}\Lambda_{A_{\Diamond}}P_{\Diamond}^T,\quad B_{\Diamond}=P_{\Diamond}\Lambda_{B_{\Diamond}}P_{\Diamond}^T, \quad{\Diamond}=x,y,z$$
with
$$ (P_{\Diamond})_{i,j} = \sin\Big(\textstyle\frac{ij\pi}{N_{\Diamond}} \Big), \quad i,j=1, \cdots, N_{\Diamond},$$
$$\Lambda_{A_{\Diamond}} = \text{diag}(\lambda^{A_{\Diamond}}_1,\cdots,\lambda^{A_{\Diamond}}_{N_{\Diamond}}),\quad\lambda^{A_{\Diamond}}_i = \textstyle\frac {h_{\Diamond}}{6} \Big(6-4\sin\Big(\frac{(i-1)\pi}{2N_{\Diamond}}\Big)\Big),\quad i=1, \cdots, N_{\Diamond},$$
$$\Lambda_{B_{\Diamond}} = \text{diag}(\lambda^{B_{\Diamond}}_1,\cdots,\lambda^{B_{\Diamond}}_{N_{\Diamond}}),\quad\lambda^{B_{\Diamond}}_i = \textstyle\frac {4}{h_{\Diamond}} \Big(\sin^2\Big(\frac{(i-1)\pi}{2N_{\Diamond}}\Big) \Big),\quad i=1, \cdots, N_{\Diamond}.$$
The other steps of implementation are similar to the ones as in the Dirichlet case,  but note that
the transformation from $U$ to $\widetilde{U}$ (resp. from $\widetilde{U}$ to $U$) in the case can be implemented by applying a Discrete Fourier Transform (DFT) (resp. an inverse DFT, abbreviated as iDFT) in each  dimension separately.

\section{Numerical experiments}\label{numerical}

In this section we will present some numerical experiments to verify the error estimates obtained in Section \ref{semilinear_theory} and demonstrate the performance of the EIFE method. All tests are done using Matlab on a Laptop with  Intel i5-8250U, 1.80GHz CPU and 8GB memory. Specifically, we choose
the EIFE1 scheme (\ref{Euler}) and the EIFE2 scheme \eqref{second-order} with $c_2=\frac 1 2$.

\subsection{Convergence tests}
We  verify the error estimates obtained in Theorem \ref{thm7} for EIFE1 scheme and Theorem \ref{thm4} for EIFE2 scheme, and numerical errors $\|u(t_n)-u_h^n\|_0$ and $\|u(t_n)-u_h^n\|_1$  are all evaluated at the terminal  time $ T $.

\begin{example}\label{ex2}
	In this example, we consider the following two-dimensional linear reaction-diffusion problem with homogeneous Dirichlet boundary condition:
	\begin{equation*}
	\left\{\begin{split}
	&u_t= \frac 1 2\Delta u - \frac 1 2 \pi^2 u + \frac 1 2 \pi^2 e^{-\pi^2 t}\sin(\pi x)\sin(\pi y), \quad (x,y) \in \Omega,\,0 \leq t \leq T\\
	&u(0, x, y) = (\sin(\pi x)-1)\sin(\pi y), \quad (x,y) \in \Omega, \\
	\end{split}  \right.
	\end{equation*}
	where $\Omega = (\frac 1 2, \frac 5 2) \times (0, 1)$ and the terminal time $T=1$. The exact solution is given by $u(t,x,y,z)=e^{-\pi^2t}(\sin(\pi x)-1)\sin(\pi y)$.
\end{example}

For the spatial accuracy tests, we run the EIFE2 scheme with fixed $N_T=1024$ (i.e., $\Delta t =T/N_T = 1/1024$) and uniformly refined spatial meshes with  $N_x\times N_y = 8\times 4$, $16 \times 8$, $32 \times 16$ and $64 \times 32$, respectively, so that the spatial mesh sizes are much coarser compared to the time step size.
For the temporal accuracy tests, we run the EIFE1 and EIFE2 schemes with fixed $N_x\times N_y = 2048\times 1024$ and uniform time step with $N_T= 16,32,64,128$.
Also, the overall cost per time step of EIFE method is tested with fixed $N_T=50$ and uniform spatial meshes with $N_x\times N_y=512 \times 256$, $1024 \times 512$, $2048\times 1024$ and $4096\times 2048$. Since the analysis of computing cost of EIFE1 scheme is same as EIFE2 scheme, we only test the running time with EIFE2 scheme.
All numerical results are reported in Table \ref{tab5}, including the solution errors  measured in the $L^2$ and $H^1$ norms and corresponding convergence rates. We observe the roughly second-order spatial convergence with respect to  both the  $L^2$ and $H^1$ norms, so the convergence order is one-order higher   than the  expected value  one in terms of the $H^1$-norm in this case. It is also easy to find  the first-order temporal convergence for the EIFE1 scheme and the  second-order temporal convergence for the EIFE2 scheme in both the  $L^2$ and $H^1$ norms, which coincide very well  with the error estimates derived in Theorems \ref{thm7} and  \ref{thm4}.
Table \ref{tab7} reports the average  CPU time costs  (seconds) per step for the EIFE2 scheme and  corresponding growth factors along the refinement of the spatial mesh.
The results clearly show that the computational cost grows almost linearly along with the number of mesh nodes, which matches well with
the property of FFT and demonstrates the high efficiency of our EIFE method.

\begin{table}[!ht]
	\renewcommand{\arraystretch}{1.0}\small
	\caption{Numerical results on the solution errors measured  in the $L^2$ and $H^1$ norms and corresponding convergence rates  for the EIFE1 and EIFE2 schemes in Example \ref{ex2}.  }
	\centering
	\begin{tabular}{|cccccc|}
		\hline
		$N_T$ & $N_x \times N_y$ & $\|u_h^n-u(t_n)\|_0$ & CR &  $\|u_h^n-u(t_n)\|_1$ & CR \\
		\hline
		\multicolumn{6}{|c|}{Spatial accuracy  tests  for EIFE2}\\
		\hline
		1024 & $8 \times 4$ & 2.1975e-05 & - &  5.8018e-05 & - \\
		1024 & $16 \times 8$ & 6.8220e-06 & 1.69 & 2.0817e-05 & 1.48 \\
		1024 & $32 \times 16$ & 1.8046e-06& 1.92 & 6.2344e-06& 1.74 \\
		1024 & $64 \times 32$ & 4.5693e-07 & 1.98 & 1.6276e-06 & 1.94 \\
		\hline
		\multicolumn{6}{|c|}{Temporal accuracy  tests  for EIFE1}\\
		\hline
		16 & $2048 \times 1024$ & 1.6807e-05 & - & 6.1597e-05 & - \\
		32 & $2048 \times 1024$ & 9.2840e-06 & 0.86 & 3.3985e-05 & 0.86 \\
		64 & $2048 \times 1024$ & 4.8687e-06 & 0.93 & 1.7814e-05 & 0.93 \\
		128 & $2048 \times 1024$ & 2.4919e-06 & 0.97 & 9.1158e-06 & 0.97\\
		\hline
		\multicolumn{6}{|c|}{Temporal accuracy  tests  for EIFE2}\\
		\hline
		16 & $2048 \times 1024$ & 8.5334e-06 & - & 3.1234e-05 & -\\
		32 & $2048 \times 1024$ & 1.6087e-06 & 2.41 & 5.8873e-06 & 2.41\\
		64 & $2048 \times 1024$ & 3.5955e-07 & 2.16 & 1.3157e-06 & 2.16\\
		128 & $2048 \times 1024$ & 8.4991e-08 & 2.08 & 3.1121e-07 & 2.08\\
		\hline
	\end{tabular}
	\label{tab5}
\end{table}

\begin{table}[!ht]
	\renewcommand{\arraystretch}{1.0}\small
	\caption{The average  CPU time costs  (seconds) per step under different spatial meshes and  corresponding growth factors with respect to the number of mesh nodes  for the EIFE2 scheme in Example \ref{ex2}.}
	\centering
	\begin{tabular}{|ccc|}
		\hline
		$N_x \times N_y$  & Average CPU time  & Growth\\
		&cost per step&factor \\
		\hline
		$512 \times 256$ & 1.121 & - \\
		$1024 \times 512$  & 4.444 & 0.99 \\
		$2048 \times 1024$ & 17.516 & 0.99 \\
		$4096 \times 2048$  & 75.944 & 1.06 \\
		\hline
	\end{tabular}
	\label{tab7}
\end{table}

\begin{example}\label{ex3}
	In this example, we consider the traveling wave problem governed by the
	following 3D Allen-Cahn equation with double-well potential function:
	\begin{equation*}
	\left\{\begin{split}
	&u_t= \Delta u -\frac{1}{\epsilon^2} (u^3-u), \quad (x,y,z) \in \Omega ,\; 0 \leq t \leq T,\\
	&u(0, x, y, z) = \frac 1 2 (1 - \tanh(\frac{x}{2\sqrt{2} \epsilon})), \quad (x,y,z) \in \Omega, \\
	\end{split}  \right.
	\end{equation*}
	where $\Omega = (0, \sqrt{2}) \times (0, \frac 1 8) \times (0, \frac 1 8)$. 	
	The exact solution is given by $u(t,x,y,z)=\frac 1 2 (1-\tanh (\frac{x-st}{2\sqrt{2}\epsilon}))$, where $s = \frac{3}{\sqrt{2}\epsilon}$, and the Dirichlet boundary condition is correspondingly imposed, which is clearly nonhomogeneous. The terminal time is taken to be $T=\frac{3\sqrt{2} \epsilon}{5}$.
\end{example}

We set the interface thickness parameter $\epsilon=0.05$.
For the spatial accuracy tests, we run the EIFE2 scheme with fixed $N_T=1024$ and uniformly refined spatial meshes with  $N_x\times N_y = 64\times 4\times 4$, $128\times 8\times 8$, $256\times 16\times 16$ and  $512\times 32\times 32$, respectively, so that the spatial mesh sizes are much coarser compared to the time step size.
For the temporal accuracy tests,  we run the EIFE1 and  EIFE2 schemes with fixed $N_x\times N_y =  1024\times 64\times 64$ and uniformly  refined time step sizes with
$N_T= 16,32,64$ and $128$. All numerical results are reported in Table \ref{tab6}, including the errors measured in the  $L^2$ and $H^1$ norms and corresponding convergence rates. We still observe the roughly second-order spatial convergence in the  $L^2$-norm but it drops to between 1.11 to 1.28 (but still higher than one)
in the $H^1$-norm. It is also  seen that  the temporal convergence is just slightly higher than one for the EIFE1 scheme and  two for the EIFE2 scheme, which basically match the theoretical results.

\begin{table}[!ht]
	\renewcommand{\arraystretch}{1.0}\small
	\centering
	\caption{Numerical results on the solution errors measured  in the $L^2$ and $H^1$ norms and corresponding convergence rates for the EIFE1 and EIFE2 schemes in Example \ref{ex3}. }
	\begin{tabular}{|cccccc|}
		\hline
		$N_T$ & $N_x\times N_y\times N_z$ & $\|u_h^n-u(t_n)\|_0$ & CR & $\|u_h^n-u(t_n)\|_1$ & CR \\
		\hline
		\multicolumn{6}{|c|}{Spatial accuracy  tests  for EIFE2 }\\
		\hline
		1024 & $64 \times 4 \times 4$ & 5.6535e-06 & - & 1.0563e-04 & - \\
		1024 & $128 \times 8 \times 8$ & 1.5135e-06 & 1.90 & 4.8833e-05 & 1.11 \\
		1024 & $256 \times 16 \times 16$ & 4.3188e-07 & 1.81 & 2.0047e-05 & 1.28 \\
		1024 & $512 \times 32 \times 32$ & 1.2193e-07 & 1.82 & 8.2719e-06 & 1.28\\
		\hline
		\multicolumn{6}{|c|}{Temporal accuracy tests for EIFE1}\\
		\hline
		16 & $1024 \times 64 \times 64$ & 2.4559e-04 & - & 3.3100e-02 & - \\
		32 & $1024 \times 64 \times 64$ & 7.5000e-05 & 1.71 & 1.0500e-02 & 1.66 \\
		64 & $1024 \times 64 \times 64$ & 2.9850e-05 & 1.33 & 4.2000e-03 & 1.32 \\
		128 & $1024 \times 64 \times 64$ & 1.2765e-05 & 1.23 & 1.9000e-03 & 1.14\\
		\hline
		\multicolumn{6}{|c|}{Temporal accuracy tests for EIFE2}\\
		\hline
		16 & $1024 \times 64 \times 64$ & 9.5466e-04 & - & 3.6900e-02 & -\\
		32 & $1024 \times 64 \times 64$ & 3.0789e-05 & 4.95 & 2.5000e-03 & 3.88\\
		64 & $1024 \times 64 \times 64$ & 4.6133e-06 & 2.74 & 4.8362e-04 & 2.37\\
		128 & $1024 \times 64 \times 64$ & 9.5582e-07 & 2.27 & 1.1057e-04 & 2.13\\
		\hline
	\end{tabular}
	\label{tab6}
\end{table}

\subsection{3D Grain coarsening simulations}
We now illustrate the performance of the proposed EIFE method through numerical simulation of the 3D grain coarsening process. In particular, the EIFE2 scheme is used for its higher-order accuracy in time.

\begin{example}\label{ex4}
	We consider the grain coarsening process governed by the following 3D Allen-Cahn equation with Flory-Huggins potential function:
	\begin{equation*}
	u_t = \epsilon^2 \Delta u + \frac{\theta}{2}\ln\frac{1-u}{1+u}+\theta_cu, \quad (x,y,z)\in \Omega, \; 0 \leq t \leq T,
	\end{equation*}
	where $\Omega = (0, 1)^3$.  The initial data is generated by random numbers on each mesh point ranging from $-0.9$ to $0.9$, and the periodic boundary condition is imposed. This equation can be regarded as the $L^2$ gradient flow of  the following  energy functional
	\begin{equation*}
	E(u) = \int_{\Omega} \frac{\theta}{2}\Big((1+u)\ln(1+u)+(1-u)\ln(1-u) \Big)-\frac{\theta_c}{2}u^2+\frac{\epsilon^2}{2}|\nabla u|^2\dd \bm x,
	\end{equation*}
	and thus the energy monotonically decays along the time.
\end{example}

We set the interface thickness parameter  $\epsilon=0.01$ and the Flory-Huggins potential parameters $\theta_c=1.6, \theta = 0.8$. This problem satisfies the  maximum bound principle with the maximum bound value $\gamma\approx 0.9575$, i.e., $|u(t,\bm x)|\leq \gamma$ for all $\bm x\in\Omega$ and $t\ge 0$ \cite{LLJF2021,ChenJing2022}. The terminal time is set to be $T=20$ for our simulation. We take the spatial mesh with $N_x=N_y=N_z=128$ which means the mesh size $h=1/128$, and the total time step size $N_T=2048$ (i.e., $\Delta t = T/N_T=5/512$). Evolutions of the supremum norm and the energy of the numerical solutions produced by the EIFE2 scheme are plotted in Fig. \ref{fig32}. We observe that  the maximum bound principle is well preserved and the energy also decays monotonically along the time. Fig. \ref{fig33} presents simulated phase structures of the numerical solutions at times $t=$ 0, 2.5, 5, 10, 15, and  20, and from which we can clearly observe the coarsening process of the 3D grains.

\begin{figure}[!ht]
	\centerline{
		\includegraphics[width = 150pt,height=120pt]{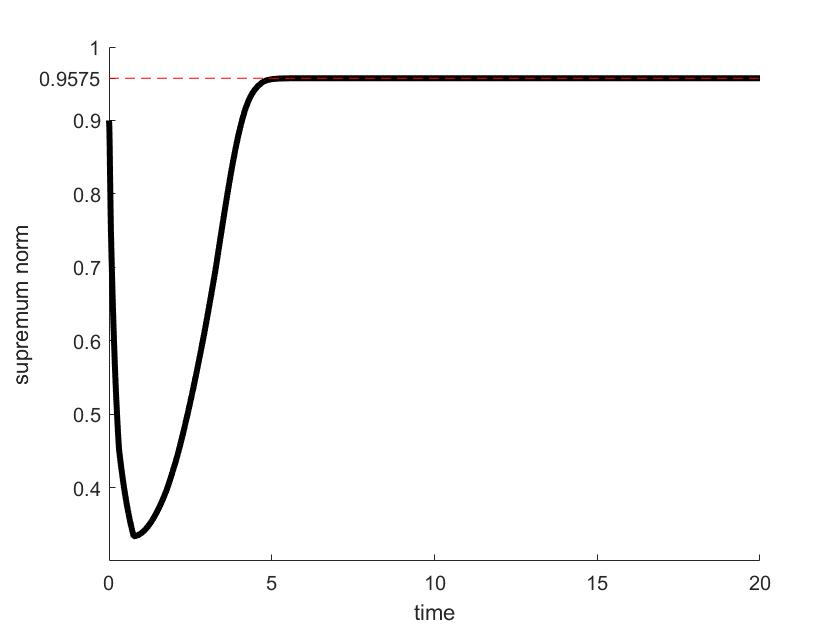}
		\includegraphics[width = 150pt,height=120pt]{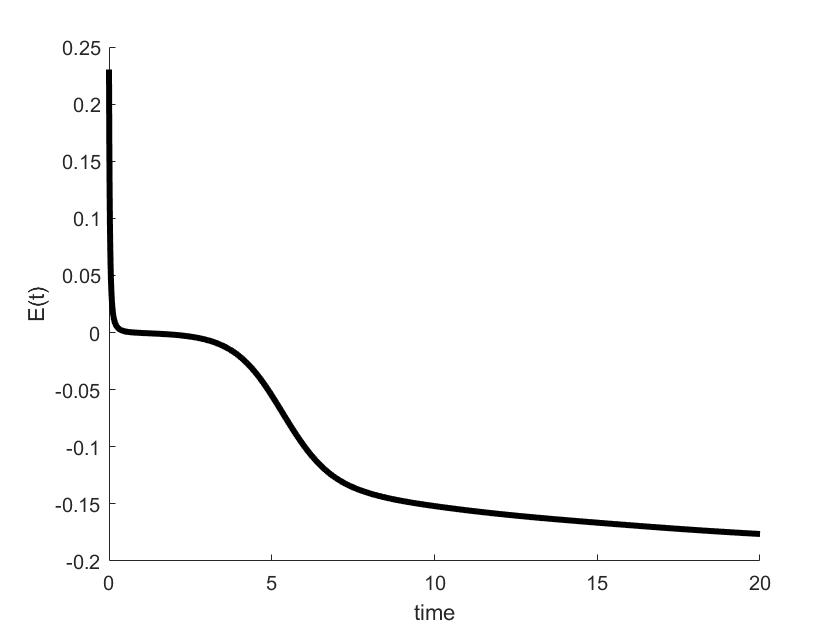}}\vspace{-0.2cm}
	\caption{The evolutions of supremum norm (left) and energy (right) of the numerical solution produced by the EIFE2 scheme for Example \ref{ex4}.}
	\label{fig32}
\end{figure}

\begin{figure}[!ht]
	\centerline{
		\includegraphics[width = 120pt,height=120pt]{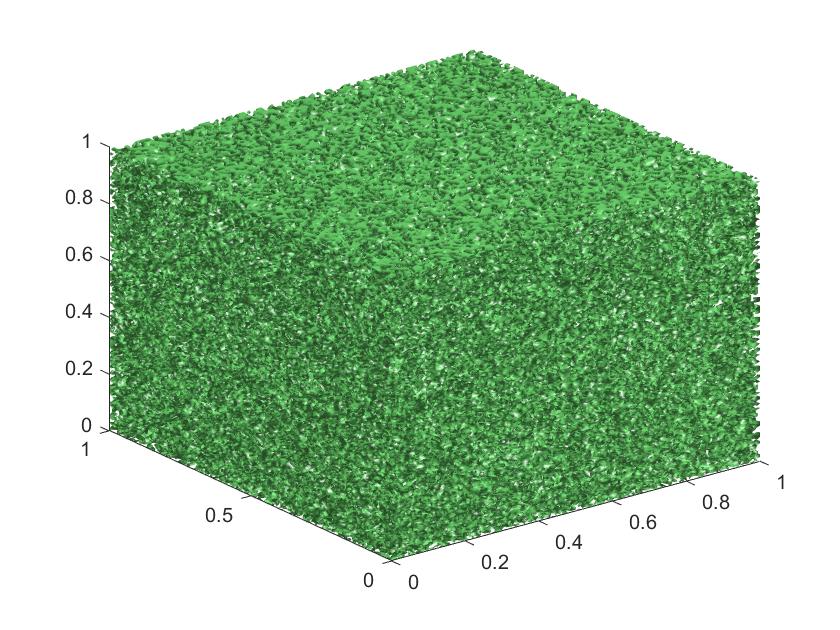}		
		\includegraphics[width = 120pt,height=120pt]{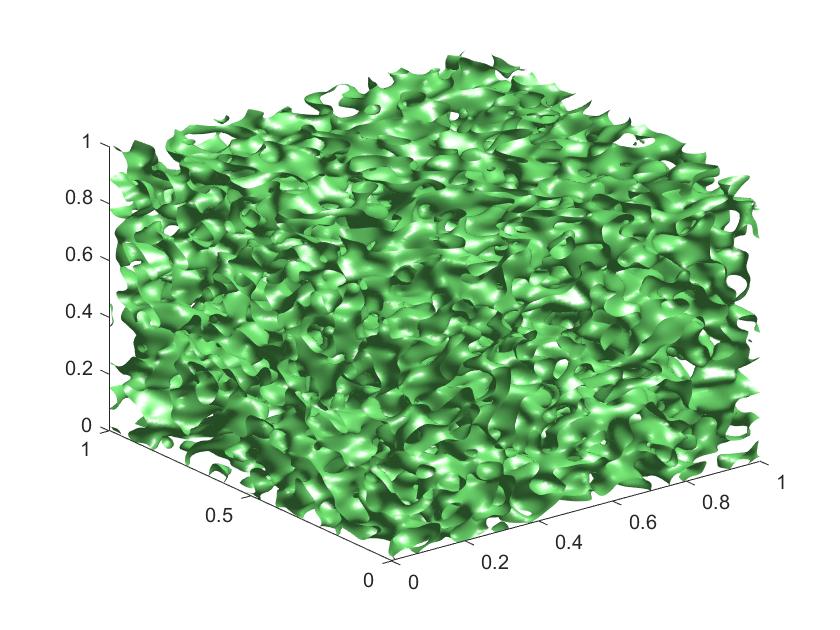}
		\includegraphics[width = 120pt,height=120pt]{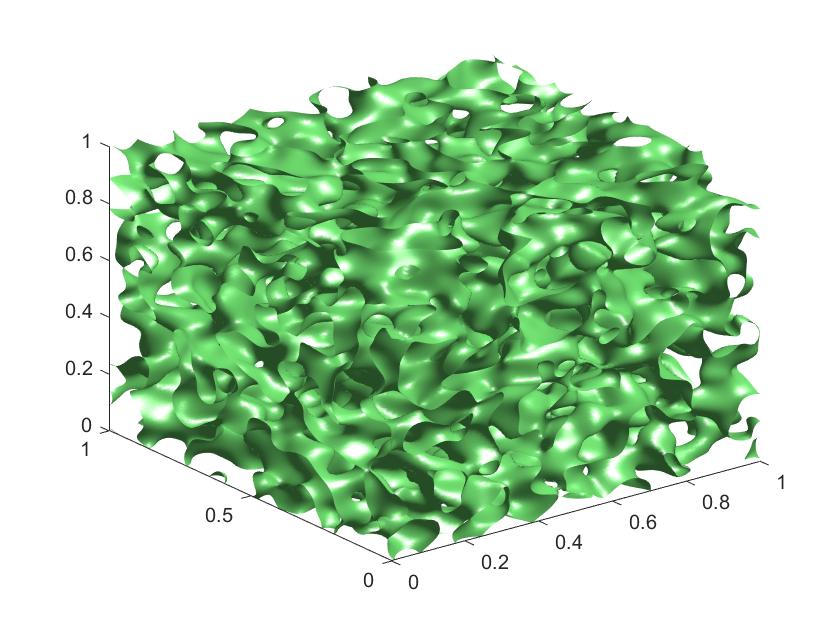}
	}\vspace{-0.2cm}
	\centerline{
		\includegraphics[width = 120pt,height=120pt]{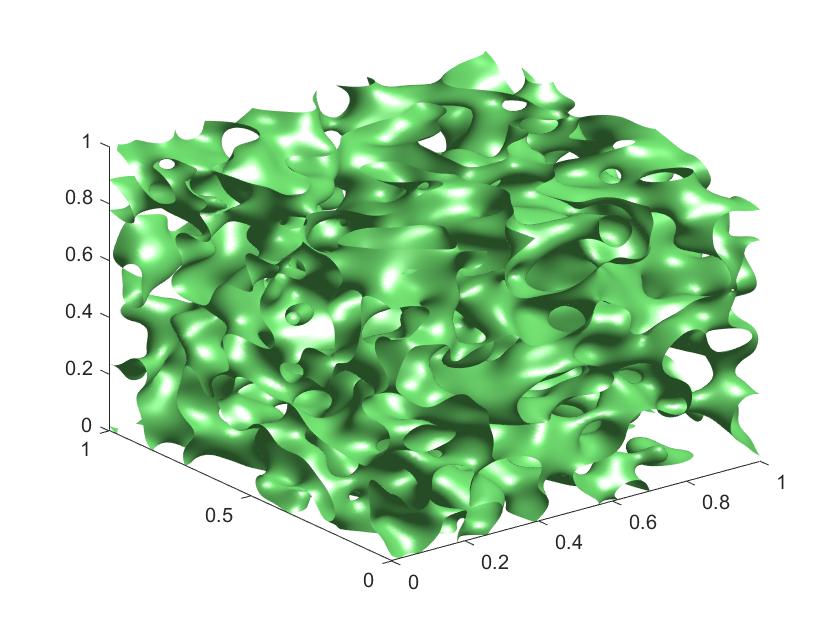}
		\includegraphics[width = 120pt,height=120pt]{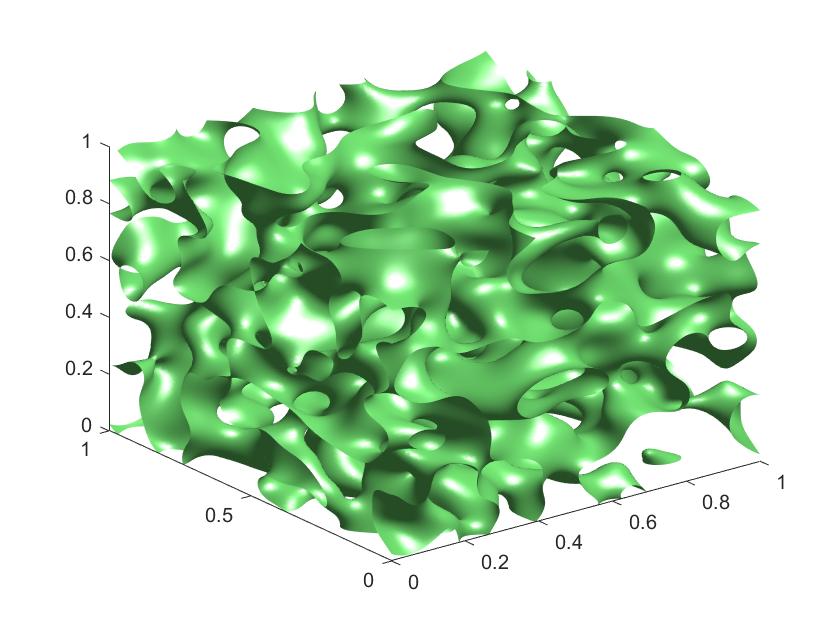}
		\includegraphics[width = 120pt,height=120pt]{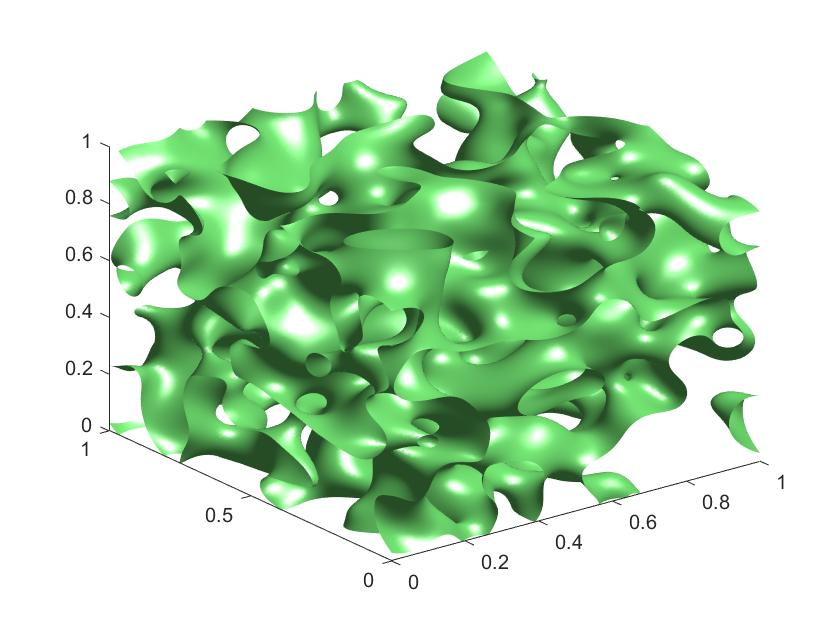}
	}\vspace{-0.2cm}
	\caption{Simulated 3D phase structures at $t=0, 2.5, 5, 10, 15, 20$ (from left to right and top to bottom) produced by the EIFE2 scheme for Example \ref{ex4}.}
	\label{fig33}
\end{figure}

\section{Conclusions}\label{seccon}

In this paper, we develop an efficient EIFE method for solving a class of semilinear parabolic equations taking the form \eqref{eq1-1} in regular domains, in which
the fully-discrete solution is  obtained by using first the finite element method for spatial discretization and then explicit exponential Runge-Kutta approximation for temporal integration. The EIFE method allows for fast implementation based on FFT and tensor product spectral decomposition. We  successfully
derive optimal error estimates in the $H^1$-norm for the EIFE method with one and two RK stages  when
the Dirichlet boundary condition s imposed. Some numerical examples are also presented to demonstrate the accuracy and high efficiency of the proposed method.
Rigorous error analysis of the EIFE method in the $L^2$-norm and  for the model problem with periodic boundary condition still remain to be explored.
In addition, the numerical method and corresponding error analysis framework developed in this paper also naturally enable us to further investigate the localized ETD methods \cite{LiJu2021, HoangJu2018} with solid theoretical support.


\bibliographystyle{abbrv}
\bibliography{ref}
\end{document}